\documentclass[reqno]{amsart}
\usepackage{amsmath,amssymb,amsthm,graphicx,a4wide,enumerate}
\usepackage[small,bf]{caption}
\theoremstyle{plain}
\newtheorem{theorem}{Theorem}
\newtheorem{lemma}[theorem]{Lemma}
\newtheorem{corollary}[theorem]{Corollary}
\theoremstyle{definition}
\newtheorem{definition}{Definition}
\theoremstyle{remark}
\newtheorem{remark}{Remark}
\newtheorem{example}{Example}
\newcommand{\prn}[1]{\left(#1\right)}
\newcommand{\abs}[1]{\left|#1\right|}
\newcommand{\norm}[1]{\left\|#1\right\|}
\newcommand{\pd}[2]{\frac{\partial#1}{\partial#2}}
\newcommand{\ud}[1]{\, \mathrm{d}#1}
\newcommand{\sfrac}[2]{{}^{#1}\!/_{#2}}
\newcommand{\phiIC}{\Phi}
\newcommand{\xfoot}{\vec{x}_{\text{foot}}}
\newcommand{\phiSP}{\psi}

\newcommand{\2}{\frac{1}{2}}

\begin{document}
\parskip.9ex

\title[Jet Schemes for Advection Problems]
{Jet Schemes for Advection Problems}
\author{Benjamin Seibold}
\address[Benjamin Seibold]
{Department of Mathematics \\ Temple University \\ \newline
1805 North Broad Street \\ Philadelphia, PA 19122}
\email{seibold@temple.edu}
\urladdr{http://www.math.temple.edu/\~{}seibold}
\author{Rodolfo R. Rosales}
\address[Rodolfo R. Rosales]
{Department of Mathematics \\ Massachusetts Institute of Technology \\
\newline 77 Massachusetts Avenue \\ Cambridge, MA 02139}
\email{rrr@math.mit.edu}
\author{Jean-Christophe Nave}
\address[Jean-Christophe Nave]
{Department of Mathematics and Statistics \\ McGill University \\
\newline 805 Sherbrooke W. \\ Montreal, QC, H3A 2K6, Canada}
\email{jcnave@math.mcgill.ca}
\urladdr{http://www.math.mcgill.ca/jcnave}
\begin{abstract}
We present a systematic methodology to develop high order accurate numerical approaches for linear advection problems. These methods are based on evolving parts of the jet of the solution in time, and are thus called jet schemes. Through the tracking of characteristics and the use of suitable Hermite interpolations, high order is achieved in an optimally local fashion, i.e.\ the update for the data at any grid point uses information from a single grid cell only. We show that jet schemes can be interpreted as advect--and--project processes in function spaces, where the projection step minimizes a stability functional. Furthermore, this function space framework makes it possible to systematically inherit update rules for the higher derivatives from the ODE solver for the characteristics. Jet schemes of orders up to five are applied in numerical benchmark tests, and systematically compared with classical WENO finite difference schemes. It is observed that jet schemes tend to possess a higher accuracy than WENO schemes of the same order.
\end{abstract}

\maketitle

\section{Introduction} \label{sec:introduction}
%
In this paper we consider a class of approaches for linear advection
problems that evolve parts of the jet of the solution in time.
Therefore, we call them \emph{jet schemes}.\footnote{The term ``jet
   scheme'' exists in the fields of algebra and algebraic geometry,
   introduced by Nash \cite{Nash1995} and popularized by Kontsevich
   \cite{Kontsevich1995}, as a concept to understand singularities.
   Since we are here dealing with a computational scheme for a partial
   differential equation, we expect no possibility for confusion.}
As we will show, the idea of tracking derivatives in addition to function
values yields a systematic approach to devise high order accurate numerical
schemes that are very localized in space. The results presented are a
generalization of gradient-augmented schemes (introduced in
\cite{NaveRosalesSeibold2010}) to arbitrary order.

In this paper we specifically consider the linear advection equation
\begin{equation} \label{eq:linear_advection}
\phi_t+\vec{v}\cdot\nabla\phi = 0
\end{equation}
for $\phi\/$, with initial conditions $\phi(\vec{x},0)=\phiIC(\vec{x})\/$.
Furthermore, we assume that the (given) velocity field
$\vec{v}(\vec{x},t)\/$ is smooth. Equation~\eqref{eq:linear_advection}
alone is rarely the central interest of a computational
project. However, it frequently occurs as a part of a larger problem. One
such example is the movement of a front under a given velocity field (in
which case $\phi\/$ would be a level set function \cite{OsherSethian1988}).
Another example is that of advection-reaction, or advection-diffusion,
problems solved by fractional steps. Here, we focus solely on the advection
problem itself, without devoting much attention to the background in which
it may arise. However, we point out that generally one cannot, at any
time, simply track back the solution to the initial data. Instead, the
problem background generally enforces the necessity to advance the
approximate solution in small time increments $\Delta t\/$.

The accurate (i.e.~high-order) approximation of \eqref{eq:linear_advection}
on a fixed grid is a non-trivial task. Commonly used approaches can be put
in two classes. One class comprises finite difference and finite volume
schemes, which store function values or cell averages. These methods
achieve high order accuracy by considering neighborhood information, with
wide stencils in each coordinate direction. Examples are
ENO~\cite{ShuOsher1988} or WENO~\cite{LiuOsherChan1994} schemes with
strong stability preserving (SSP)
time stepping~\cite{GottliebShu1998,Shu1988,ShuOsher1988}
or flux limiter approaches \cite{VanLeer1973,VanLeer1979}. Due to their
wide stencils, achieving high order accurate approximations near
boundaries can be challenging with these methods. In addition, the
generalization of ENO/WENO methods to adaptive grids (quadtrees, octrees)
\cite{BergerOliger1984,MinGibou2007} is non-trivial. The other class of
approaches is comprised by semi-discrete methods, such as discontinuous
Galerkin (DG) \cite{CockburnShu1988,HesthavenWarburton2008,ReedHill1973}.
These achieve high order accuracy by approximating the flux through cell
boundaries based on high order polynomials in each grid cell. Generally,
DG methods are based on a weak formulation of \eqref{eq:linear_advection},
and the flow between neighboring cells is determined by a numerical flux
function. All integrals over cells and cell boundaries are
approximated by Gaussian quadrature rules, and the time stepping is done
by SSP Runge-Kutta schemes \cite{GottliebShu1998,Shu1988,ShuOsher1988}.
While the DG formalism can in principle yield any order of accuracy, its
implementation requires some care in the design of the data
structures (choice of polynomial basis, orientation of normal vectors,
etc.). Furthermore, SSP Runge-Kutta schemes of an order higher than 4
are quite difficult to design \cite{Gottlieb2005,GottliebKetchesonShu2009}.

The approaches (jet schemes) considered here fall into the class of
semi-Lagrangian approaches: they evolve the numerical solution on a fixed
Eulerian grid, with an update rule that is based on the method of
characteristics. Jet schemes share
some properties with the methods from both classes described
above. Like DG methods, they are based on a high order polynomial
approximation in each grid cell, and store derivative information in
addition to point values. However, instead of using a weak formulation,
numerical flux functions, Gaussian quadrature rules, and SSP schemes,
here characteristic curves are tracked --- which can be done with simple,
non-SSP, Runge-Kutta methods.
Furthermore, unlike semi-discrete methods, jet schemes construct
the solution at the next time step in a completely local fashion
(point by point). Each of the intermediate stages of a Runge-Kutta
scheme does not require the reconstruction of an approximate
intermediate representation for the solution in space. This
property is shared with Godunov-type finite volume methods. Fundamental
differences with finite volume methods are that jet schemes are not
conservative by design, and that higher derivative information is
stored, rather than reconstructed from cell averages. Finally, like many
other semi-Lagrangian approaches, jet schemes treat boundary conditions
naturally (the distinction between an ingoing and an outgoing
characteristic is built into the method), and they do not possess a
Courant-Friedrichs-Lewy (CFL)
condition that restricts stability. This latter property may be of
relevance if the advection equation~\eqref{eq:linear_advection} is one
step in a more complex problem that exhibits a separation of time scales.

For advection problems \eqref{eq:linear_advection}, jet schemes are
relatively simple and natural approaches that yield high order accuracy
with optimal locality: to update the data at a given grid point,
information from only a single grid cell is used
\cite{NaveRosalesSeibold2010}. In addition, the high order polynomial
approximation admits the representation of certain structures of subgrid
size --- see \cite{NaveRosalesSeibold2010} for more details.

Jet schemes are based on the idea of advect--and--project: one time step
in the solution of Equation~\eqref{eq:linear_advection} takes the form
\begin{equation} \label{intro:advect:and:project}
 \phi^{n+1} = P \circ A_{t+\Delta t\/,\,t}\,\phi^n\/,
\end{equation}
where: (i) $\phi^m\/$ denotes the solution at time $t_m\/$ --- with
$t_{n+1}=t_n+\Delta t\/$, (ii) $A_{t+\Delta t\/,\,t}\/$ is an approximate
advection operator, as obtained by evolving the solution along
characteristics using an appropriate ODE solver, and (iii) $P\/$ is a
projection operator, based on knowing a specified portion of the jet of
the solution at each grid point in a fixed cartesian grid. At the end of
the time step, the solution is represented by an appropriate cell based,
polynomial Hermite interpolant --- produced by the projection $P\/$. We
use polynomial Hermite interpolants because they have a very useful
stabilizing property: in each cell, the interpolant is a polynomial
minimizer for the $L^2\/$ norm of a certain high derivative of the
interpolated function. This controls the growth of the derivatives of
the solution (e.g.: oscillations), and thus ensures stability.

Equation~\eqref{intro:advect:and:project} is not quite a numerical scheme,
since advecting the complete function $\phi^n\/$ would require a continuum
of operations. However, it can be easily made into one. In order to be able
to apply $P\/$, all we need to know is the values of some partial derivatives
of $A_{t+\Delta t\/,\,t}\,\phi^n\/$ (the portion of the jet that $P\/$ uses)
at the grid points. These can be obtained from the ODE solver as
follows: consider the formula (provided by using the ODE solver along
characteristics) that gives the value of $A_{t+\Delta t\/,\,t}\,\phi^n\/$ at
any point (in particular, the grid points). Then take the appropriate
partial derivatives of this formula; this gives an update rule that can
be used to obtain the required data at the grid points.
This process provides an implementable scheme that is fully equivalent
to the continuum (functional) equation~\eqref{intro:advect:and:project}.
We call such a scheme \emph{superconsistent}, since it maintains
functional consistency between the function and its partial derivatives.

Finally, we point out that no special restrictions on the ODE solver
used are needed. This follows from the minimizing property of the
Hermite interpolants mentioned earlier, and superconsistency --- which
guarantees that the property is not lost, as the time update occurs
in the functional sense of Equation~\eqref{intro:advect:and:project}.
Thus, for example, regular Runge-Kutta schemes can be used with
superconsistent jet schemes --- unlike WENO schemes, which require
SSP ODE solvers to ensure total variation diminishing (TVD)
stability~\cite{GottliebShuTadmor2001}.

This paper is organized as follows.
The polynomial representation of the approximate solution is presented in
\S~\ref{sec:interp_and_project}. There we show how, given suitable parts
of the jet of a smooth function, a cell-based Hermite interpolant can be
used to obtain a high order accurate approximation. This interpolant gives
rise to a projection operator in function spaces, defined by evaluating
the jet of a function at grid points, and then constructing the piecewise
Hermite interpolant.
In \S~\ref{sec:Advect:Update:Time}, the jet schemes' advect--and--project
approach sketched above is described in detail. In particular, we show
that superconsistent jet schemes are equivalent to advancing the solution
in time using the functional equation~\eqref{intro:advect:and:project}.
This interpretation is then used to systematically inherit update rules
for the solution's derivatives, from the numerical scheme used for the
characteristics. Specific two-dimensional schemes of orders $1\/$,
$3\/$, and $5\/$ are constructed.
These are then investigated numerically in \S~\ref{sec:numerical_results}
for a benchmark test, and compared with classical WENO schemes of the
same orders.
Boundary conditions and stability are discussed in
\S~\ref{subsec:Boundary:Conditions} and
\S~\ref{subsec:adv:funspaces}, respectively.

\section{Interpolations and Projections} \label{sec:interp_and_project}
%
In this section we discuss the class of projections $P\/$ that are
used and required by jet schemes --- see
Equation~\eqref{intro:advect:and:project}.
We begin, in \S~\ref{subsec:generalized_hermite_interpolation}, by
presenting the cell-based generalized Hermite interpolation problem
of arbitrary order in any number of dimensions.
In \S~\ref{subsec:global_interpolant} we construct, on an arbitrary
cartesian grid, global interpolants --- using the cell-based generalized
Hermite interpolants of \S~\ref{subsec:generalized_hermite_interpolation}.
Then we introduce a stability functional,
which is the key to the stability of the superconsistent jet schemes.
Next, in \S~\ref{subsec:total:partial:jets}, two different types of
portions of the jet of a function (corresponding to different kinds of
projections) are defined. These are the total and the partial $k$-jets.
In \S~\ref{subsec:projections_function_spaces}
the global interpolants defined in \S~\ref{subsec:global_interpolant} are
used to construct (in appropriate function spaces) the projections which
are the main purpose of this section.
Finally, in \S~\ref{subsec:optimally:local} we discuss possible ways to
decrease the number of derivatives needed by the Hermite interpolants,
using cell based finite differences; the notion of an optimally local
projection is introduced there.

\subsection{Cell-Based Generalized Hermite Interpolation}
\label{subsec:generalized_hermite_interpolation}
%
We start with a review of the generalized Hermite interpolation problem
in one space dimension. Consider the unit interval $[0,1]\/$. For each
boundary point $q\in\{0,1\}\/$ let a vector of data
$\prn{\phi_0^q,\phi_1^q,\dots,\phi_k^q}\/$ be given, corresponding to all
the derivatives up to order $k\/$ of some sufficiently smooth function
$\phi=\phi(x)\/$ --- the zeroth order derivative is the function itself.
Namely
\begin{equation*}
\phi_{\alpha}^{q} = \prn{\tfrac{d}{d\/x}}^\alpha\phi(q)
\quad \forall \, q\in\{0,1\}\/,\;\alpha\in\{0,1,\dots,k\}\/.
\end{equation*}
In the class of polynomials of degree less than or equal to $n=2k+1\/$,
this equation can be used to define an interpolation problem with a unique
solution. This solution, the $n^{th}$ order Hermite interpolant, can be
written as a linear superposition
\begin{equation*}
\mathcal{H}_n(x) = \sum_{q\in\{0,1\}}\;\;\sum_{\alpha\in\{0,\dots,k\}}
\phi_{\alpha}^q\;w_{n\/,\,\alpha}^q(x)
\end{equation*}
of basis functions $w_{n\/,\,\alpha}^q$, each of which solves the
interpolation problem
\begin{equation*}
\prn{\tfrac{d}{d\/x}}^{\alpha'}\! w_{n\/,\,\alpha}^q(q') =
\delta_{\alpha\/,\,\alpha'}\,\delta_{q\/,\,q'} \quad
\forall\,q'\in\{0,1\}\/,\;\alpha'\in\{0\/,\,\dots\/,\,k\}\/,
\end{equation*}
where $\delta\/$ denotes Kronecker's delta. Hence, each of the
$2(k+1) = n+1\/$ basis polynomials equals $1\/$ on exactly one boundary
point and for exactly one derivative, and equals $0\/$ for any other
boundary point or derivative up to order $k\/$. Notice that
\begin{equation*}
w_{n\/,\,\alpha}^0(x) = (-1)^\alpha\,w_{n\/,\,\alpha}^1(1-x)\/.
\end{equation*}
%
\begin{example}
The three lowest order cases of generalized Hermite interpolants are given
by the following basis functions:
\begin{itemize}
 \item linear ($k=0\/$, $n=1\/$):
  \begin{align*}
   w_{1\/,\,0}^1(x) &= x\/,
  \end{align*}
 \item cubic ($k=1\/$, $n=3\/$):
  \begin{align*}
   w_{3\/,\,0}^1(x) &= 3x^2-2x^3\/, \\
   w_{3\/,\,1}^1(x) &= -x^2+x^3\/,
  \end{align*}
 \item quintic ($k=2\/$, $n=5\/$):
  \begin{align*}
   w_{5\/,\,0}^1(x) &= 10x^3-15x^4+6x^5\/, \\
   w_{5\/,\,1}^1(x) &= -4x^3+7x^4-3x^5\/, \\
   w_{5\/,\,2}^1(x) &= \tfrac{1}{2}x^3-x^4+\tfrac{1}{2}x^5\/.
  \end{align*}
\end{itemize}
\end{example}
%
One dimensional Hermite interpolation can be generalized naturally to
higher space dimensions by using a tensor product approach, as described
next.

In $\mathbb{R}^p\/$, consider a \emph{$p$-rectangle} (or simply ``cell'')
$[a_1\/,\,b_1]\times\dots\times [a_p\/,\,b_p]\/$. Let
$\Delta x_i = b_i-a_i\/$, $1 \leq i \leq p\/$, denote the edge lengths of
the $p\/$-rectangle, and call $h = \max_{i=1}^p\Delta x_i$ the
\emph{resolution}. In addition, we use the classical multi-index notation.
For vectors $\vec{x}\in\mathbb{R}^p\/$ and
$\vec{a} \in \mathbb{N}_0^p\/$, define:
(i)   $\abs{\vec{a}} = \sum_{i=1}^p a_i\/$,
(ii)  $\vec{x}^{\,\vec{a}} = \prod_{i=1}^p x_i^{a_i}\/$, and
(iii) $\partial^{\,\vec{a}} = \partial_1^{a_1}\,\dots\,\partial_p^{a_p}\/$,
where $\partial_i = \frac{\partial}{\partial x_i}\/$.
%
\begin{definition} \label{def:p-n_polynomial}
 A \emph{$p$-$n$ polynomial} is a $p$-variate polynomial of degree less
 than or equal to $n\/$ in each of the variables. Using multi-index
 notation, a $p$-$n$ polynomial can be written as
 \begin{equation*}
  \mathcal{H}_n(\vec{x}) = \sum_{\vec{\alpha}\in\{0\/,\,\dots\/,\,n\}^p}
  c_{\vec{\alpha}} \; \vec{x}^{\,\vec{\alpha}}\/,
 \end{equation*}
 with $(n+1)^p\/$ parameters $c_{\vec{\alpha}}\/$. Note that here
 we will consider only the case where $n\/$ is odd.
\end{definition}
%
\begin{example}
 Examples of $p$-$n$ polynomials are:
 \begin{itemize}
  \item[] \hspace*{7.0ex} \parbox{10ex}{\ $p=1\/$}
          \parbox{16ex}{$p=2\/$} \parbox{7ex}{$p=3\/$}
  \item $n=1\/$: \parbox{9ex}{\ linear,} \parbox{11ex}{bi-linear,}
                 \parbox{17ex}{and\ \  tri-linear} functions.
  \item $n=3\/$: \parbox{9ex}{\ cubic,} \parbox{11ex}{bi-cubic,}
                 \parbox{17ex}{and\ \  tri-cubic} functions.
  \item $n=5\/$: \parbox{9ex}{\ quintic,} \parbox{11ex}{bi-quintic,}
                 \parbox{17ex}{and\ \  tri-quintic} functions.
 \end{itemize}
\end{example}
%
Now let the $p$-rectangle's $2^p\/$ vertices be indexed by a vector
$\vec{q}\in\{0,1\}^p\/$, such that the vertex of index $\vec{q}\/$ is at
position $\vec{x}_{\vec{q}} =
(a_1+\Delta x_1\,q_1\/,\,\dots\/,\,a_p+\Delta x_p\,q_p)\/$.
%
\begin{definition} \label{def:data}
 For $n\/$ odd, and a sufficiently smooth function $\phi=\phi(\vec{x})\/$,
 the \emph{$n$-data} on the $p$-rectangle (defined on the vertices) is the
 set of $(n+1)^p$ scalars given by
 \begin{equation} \label{eq:data}
  \phi_{\vec{\alpha}}^{\vec{q}} = \partial^{\,\vec{\alpha}}\,
  \phi(\vec{x}_{\vec{q}})\/,
 \end{equation}
 where $\vec{q}\in\{0,1\}^p\/$ and $\vec{\alpha}\in\{0,\dots,k\}^p\/$,
 with $k=\frac{n-1}{2}\/$.
\end{definition}
%
\begin{lemma} \label{lem:p-n_polynomial_uniqueness}
 Two $p$-$n$ polynomials with the same $n$-data on some $p\/$-rectangle,
 must be equal.
\end{lemma}
%
\begin{proof}
 Let $\phi\/$ be the difference between the two polynomials, which has
 zero data. We prove that $\phi\equiv 0\/$ by induction over $p\/$.
 For $p=1\/$, we have a standard 1-D Hermite interpolation problem, whose
 solution is known to be unique. Assume now that the result applies for
 $p-1\/$. In the $p$-rectangle, consider the functions $\phi\/$,
 $\partial_p\,\phi\/$, \dots, $\partial_p^k\,\phi\/$ --- both at the
 ``bottom'' hyperface ($x_p = a_p\/$, i.e.~$q_p = 0\/$) and the ``top''
 hyperface ($x_p = b_p\/$, i.e.~$q_p = 1\/$). For each of these functions,
 zero data is given at all of the corner vertices of the two hyperfaces.
 Therefore, by the induction assumption, all of these functions vanish
 everywhere on the top and bottom hyperfaces. Consider now the ``vertical''
 lines joining a point $(x_1\/,\,\dots\/,\,x_{p-1})\in [a_1\/,\,b_1]
 \times\,\dots\,\times [a_{p-1}\/,\,b_{p-1}]$ in the bottom hyperface with
 its corresponding one on the top hyperface. For each of these lines we
 can use the $p = 1\/$ uniqueness result to conclude that $\phi = 0\/$
 identically on the line. It follows that $\phi = 0\/$ everywhere.
\end{proof}
%
\begin{theorem} \label{thm:interpolant:exists}
 For any arbitrary $n$-data ($n\/$ odd) on some $p$-rectangle, there
 exists exactly one $p$-$n$ polynomial which interpolates the data.
\end{theorem}
%
\begin{proof}
 Lemma~\ref{lem:p-n_polynomial_uniqueness} shows that there exists at
 most one such polynomial. The interpolating $p$-$n$ polynomial is
 explicitly given by
 \begin{equation} \label{eq:p-n_interpolant_constructed}
  \mathcal{H}_n(\vec{x}) = \sum_{\vec{q}\in\{0,1\}^p}\;\;
  \sum_{\vec{\alpha}\in\{0\/,\,\dots\/,\,k\}^p}
  \phi_{\vec{\alpha}}^{\vec{q}}\;W_{n\/,\,\vec{\alpha}}^{\vec{q}}(\vec{x})\/,
 \end{equation}
 where the $W_{n\/,\,\vec{\alpha}}^{\vec{q}}(\vec{x})\/$ are
 $p$-$n$ polynomial basis functions that satisfy
 \begin{equation*}
  \partial^{\,\vec{\alpha}{\,}'} W_{n\/,\,\vec{\alpha}}^{\vec{q}}\,
  (\vec{x}_{\vec{q}{\;}'}) = \delta_{\vec{\alpha}\/,\,\vec{\alpha}{\,}'}\;
  \delta_{\vec{q}\/,\,\vec{q}{\;}'} \quad\forall\,
  \vec{q}{\;}'\in\{0,1\}^p\/,\;\vec{\alpha}{\,}'\in\{0,\dots,k\}^p\/.
 \end{equation*}
 They are given by the tensor products
 \begin{equation*}
  W_{n\/,\,\vec{\alpha}}^{\vec{q}}(\vec{x}) = \prod_{i=1}^p
  (\Delta x_i)^{\alpha_i}\;w_{n\/,\,\alpha_i}^{q_i}(\xi_i)\/,
 \end{equation*}
 where $\xi_i = \frac{x_i-a_i}{\Delta x_i}\/$ is the relative coordinate
 in the $p$-rectangle, and the $w_{n\/,\,\alpha}^{q}$ are the univariate
 basis functions defined earlier.
\end{proof}
%
Next we show that the $p$-$n$ polynomial interpolant given by
Equation~\eqref{eq:p-n_interpolant_constructed} is a $(n+1)^{st}\/$
order accurate approximation to any sufficiently smooth function
$\phi\/$ it interpolates on a $p$-rectangle. For convenience, we
consider a $p$-cube with $\Delta x_1 = \dots = \Delta x_p = h\/$.
In this case, the interpolant in
\eqref{eq:p-n_interpolant_constructed} becomes
\begin{equation} \label{eq:p-n_interpolant_h}
 \mathcal{H}_n(\vec{x}) = \sum_{\vec{q}\in\{0,1\}^p}\;\;
 \sum_{\vec{\alpha}\in\{0\/,\,\dots\/,\,k\}^p}
 \phi_{\vec{\alpha}}^{\vec{q}}\; h^{\abs{\vec{\alpha}}} \;
 \prod_{i=1}^p w_{n\/,\,\alpha_i}^{q_i}(\xi_i)\/.
\end{equation}
%
\begin{lemma} \label{lem:p-n_polynomial_accuracy_data}
 Let the data determining the $p$-$n$ polynomial Hermite interpolant be
 known only up to some error. Then Equation~\eqref{eq:p-n_interpolant_h}
 yields the interpolation error
 \begin{equation*}
  \delta\mathcal{H}_n(\vec{x}) = \sum_{\vec{q}\in\{0,1\}^p}\;\;
  \sum_{\vec{\alpha}\in\{0\/,\,\dots\/,\,k\}^p}
  \prn{\prod_{i=1}^p w_{n\/,\,\alpha_i}^{q_i}(\xi_i)}\,
  h^{\abs{\vec{\alpha}}}\; \delta\phi_{\vec{\alpha}}^{\vec{q}}\/,
 \end{equation*}
 where the notation $\delta u\/$ indicates the error in some quantity
 $u\/$. In particular, if the data $\phi_{\vec{\alpha}}^{\vec{q}}\/$ are
 known with $O\prn{h^{n+1-\abs{\vec{\alpha}}}}\/$ accuracy, then
 $\delta\prn{\partial^{\,\vec{\alpha}}\mathcal{H}_n} =
  O\prn{h^{n+1-\abs{\vec{\alpha}}}}\/$.
\end{lemma}
%
\begin{theorem} \label{thm:p-n_polynomial_accuracy}
 Consider a sufficiently smooth function $\phi\/$, and let
 $\mathcal{H}_n(\vec{x})\/$ be the $p$-$n$ polynomial that interpolates
 the data given by $\phi\/$ on the vertices of a $p$-cube of size $h\/$.
 Then, everywhere inside the $p$-cube, one has
 \begin{equation} \label{eq:p-n_polynomial_error}
  \partial^{\,\vec{\alpha}}\mathcal{H}_n - \partial^{\,\vec{\alpha}}\phi =
  O\prn{h^{n+1-\abs{\vec{\alpha}}}}\/,
 \end{equation}
 where the constant in the error term is controlled by the $(n+1)^{st}\/$
 derivatives of $\phi\/$.
\end{theorem}
%
\begin{proof}
 Let $\mathcal{G}\/$ be the degree $n\/$ polynomial Taylor approximation
 to $\phi\/$, centered at some point inside the $p$-cube. Then, by
 construction:
 (i) $\partial^{\,\vec{\alpha}}\mathcal{G} - \partial^{\,\vec{\alpha}}\phi
      = O\prn{h^{n+1-\abs{\vec{\alpha}}}}\/$.
 In particular, the data for $\mathcal{G}$ on the $p$-cube is related to
 the data for $\mathcal{H}_n$ (same as the data for $\phi$) in the manner
 specified in Lemma~\ref{lem:p-n_polynomial_accuracy_data}. Thus:
 (ii) $\partial^{\,\vec{\alpha}}\mathcal{H}_G - \partial^{\,\vec{\alpha}}
      \mathcal{H}_n = O\prn{h^{n+1-\abs{\vec{\alpha}}}}\/$,
 where $\mathcal{H}_G\/$ is the $p$-$n$ polynomial that interpolates the
 data given by $\mathcal{G}\/$ on the vertices of the $p$-cube. However,
 from Lemma~\ref{lem:p-n_polynomial_uniqueness}:
 (iii) $\mathcal{H}_G = \mathcal{G}\/$, since $\mathcal{G}\/$ is a
 $p$-$n$ polynomial. From (i), (ii), and (iii)
 Equation~\eqref{eq:p-n_polynomial_error} follows.
\end{proof}
%
In conclusion, we have shown that: the $p$-$n$ polynomial Hermite
interpolant approximates smooth functions with $(n+1)^{st}$ order accuracy,
and each level of differentiation lowers the order of accuracy by one
level. Further: in order to achieve the full order accuracy, the data
$\phi_{\vec{\alpha}}^{\vec{q}}\/$ must be known with accuracy
$O\prn{h^{n+1-\abs{\vec{\alpha}}}}\/$.

\subsection{Global Interpolant} \label{subsec:global_interpolant}
%
Consider a rectangular computational domain $\Omega\subset \mathbb{R}^p\/$
in $p\/$ spacial dimensions, equipped with a regular rectangular grid.
Assume that at each grid point $\vec{x}_{\vec{m}}\/$, labeled by
$\vec{m}\in\mathbb{Z}^p\/$, a vector of data
$\phi_{\vec{\alpha}}^{\vec{m}}\/$ is given, where
$\vec{\alpha}\in\{0\/,\,\dots\/,\,k\}^p\/$ --- for some
$k\in\mathbb{N}_0\/$. Given this grid data, we define a global interpolant
$\mathcal{H}:\Omega\to\mathbb{R}\/$, which is a piece-wise $p$-$n$
polynomial (with $n=2k+1\/$). On each grid cell $\mathcal{H}\/$ is given
by the $p$-$n$ polynomial obtained using
Equation~\eqref{eq:p-n_interpolant_constructed},
with $\vec{q}\/$ related to the
grid index $\vec{m}\/$ by: $\vec{q}=\vec{m}-\vec{m}_0\/$ --- where
$\vec{m}_0\/$ is the cell vertex with the lowest values for each component
of $\vec{m}\/$. Note that $\mathcal{H}\/$ is $C^\infty\/$ inside each grid
cell, and $C^k\/$ across cell boundaries. However, in general,
$\mathcal{H}\/$ is not $C^{k+1}\/$.
%
\begin{remark} \label{rem:Ck:only}
 The smoothness of $\mathcal{H}\/$ is biased in the coordinate directions.
 All the derivatives that appear in the data vectors,
 i.e.~$\partial^{\,\vec{\alpha}}\mathcal{H}\/$ for
 $\vec{\alpha}\in\{0\/,\,\dots\/,\,k\}^p\/$, are defined everywhere in
 $\Omega\/$. Furthermore, at the grid points,
 $\partial^{\,\vec{\alpha}}\mathcal{H}(\vec{x}_{\vec{m}}) =
 \phi_{\vec{\alpha}}^{\vec{m}}\/$. In particular, all the partial derivatives
 up to order $k\/$ are defined, and continuous. However, not all the
 partial derivatives of orders larger than $k\/$ are defined. In general
 $\,\partial^{\vec{\alpha}}\,\mathcal{H}\/$, with
 $\vec{\alpha}\in\{0\/,\,\dots\/,\,k+1\}^p\/$, is piece-wise smooth ---
 with simple jump discontinuities across the grid hyperplanes which are
 perpendicular to any coordinate direction $x_{\ell}\/$ such that
 $\alpha_{\ell} = k+1\/$. Derivatives
 $\,\partial^{\vec{\alpha}}\,\mathcal{H}\/$, where $\alpha_{\ell} > k+1\/$
 for some $1 \leq \ell \leq p\/$,
 generally exist only in the sense of distributions.
\end{remark}
%
\begin{definition} \label{def:H_phi}
 Any (sufficiently smooth) function $\phi:\Omega\to\mathbb{R}\/$ defines
 a global interpolant $\mathcal{H}_{\phi}$ as follows: at each grid point
 $\vec{x}_{\vec{m}}\/$, evaluate the derivatives of $\phi\/$, as by
 Definition~\ref{def:data}, to produce a data vector. Namely:
 $\phi_{\vec{\alpha}}^{\vec{m}} = \partial^{\,\vec{\alpha}}
 \phi(\vec{x}_{\vec{m}})\ \forall\vec{\alpha}\in\{0\/,\,\dots\/,\,k\}^p\/$.
 Then, use these values as data to define $\mathcal{H}_{\phi}\/$
 everywhere.
\end{definition}
%
\begin{definition} \label{def:stab:fun}
 For any sufficiently smooth function $\phi\/$, define the
 \emph{stability functional} by
 \begin{equation} \label{eq:functional}
  \mathcal{F}[\phi] =
  \int_\Omega \prn{\partial^{\,\vec{\beta}}\phi(\vec{x})}^2\ud{\vec{x}}\/,
 \end{equation}
 where $\vec{\beta} = \vec{\beta}(k\/,\,p)\/$ is the $p$-vector
 $\vec{\beta} = (k+1,\dots,k+1)\/$. Of course, $\mathcal{F}\/$ also
 depends on $k\in\mathbb{N}_0\/$ and $\Omega \subset \mathbb{R}^p\/$,
 but (to simplify the notation) we do not display these dependencies.
\end{definition}
%
\begin{theorem} \label{thm:minimize:seminorm}
 Replacing a sufficiently smooth function $\phi\/$ by the interpolant
 $\mathcal{H}_{\phi}\/$ does not increase the stability functional:
 $\mathcal{F}[\mathcal{H}_{\phi}] \le \mathcal{F}[\phi]\/$. In fact:
 $\mathcal{H}_{\phi}\/$ minimizes $\mathcal{F}\/$, subject to the
 constraints given by the data
 $\partial^{\,\vec{\alpha}}\,\phi(\vec{x}_{\vec{m}})\/$, and the requirement
 that the minimizer should be smooth in each grid cell.
\end{theorem}
%
\begin{remark} \label{rem:minimize:not:unique}
The minimizer is not unique if $p > 1\/$. To see this, let $f_j(x_j), \ 1 \leq j \leq p\/$ be nonzero smooth functions, with $f_j\/$ and all its derivatives vanishing at the the grid points. Then $\psi = \mathcal{H}_{\phi} + \sum_{j=1}^p f_j(x_j) \neq \mathcal{H}_{\phi}\/$ has the same data as $\mathcal{H}_{\phi}\/$ and $\phi\/$, and $\mathcal{F}[\psi] = \mathcal{F}[\mathcal{H}_{\phi}]\/$ if $p > 1\/$.
\end{remark}
%
\begin{proof}
 Note that $\partial^{\,\vec{\beta}}\mathcal{H}_{\phi}\/$ exists and it is
 piece-wise smooth (it may have simple discontinuities across cell
 boundaries --- see Remark~\ref{rem:Ck:only}). Hence
 $\mathcal{F}[\mathcal{H}_{\phi}]\/$ is defined. Clearly, it is
 enough to show that $\mathcal{H}_{\phi}\/$ minimizes $\mathcal{F}\/$ in
 each grid cell $Q\/$. Define $\varphi = \phi-\mathcal{H}_{\phi}\/$. Since
 $\phi\/$ and $\mathcal{H}_{\phi}\/$ have the same data at the vertices of
 $Q\/$, $\varphi\/$ has zero data at the vertices. Thus
 \begin{equation} \label{eq:fun:proof}
  I_p = \int_Q \prn{\partial^{\,\vec{\beta}}\mathcal{H}_{\phi}(\vec{x})}
  \prn{\partial^{\,\vec{\beta}}\varphi(\vec{x})}\ud{\vec{x}} = 0\/,
 \end{equation}
 as shown in Lemma~\ref{lem:minimize:seminorm}. From this equality it
 follows that $\mathcal{F}[\phi] = \mathcal{F}[\mathcal{H}_{\phi}] +
 \mathcal{F}[\varphi]\/$. Now, since $\mathcal{F}[\varphi] \geq 0\/$,
 $\mathcal{F}[\mathcal{H}_{\phi}] \le \mathcal{F}[\phi]\/$.
 For any other $\phi^*\/$ satisfying the theorem statement's constraints
 for the minimizing class, $\mathcal{H}_{\phi} = \mathcal{H}_{\phi^*}\/$.
 Hence $\mathcal{F}[\mathcal{H}_{\phi}] =
 \mathcal{F}[\mathcal{H}_{\phi^*}] \leq \mathcal{F}[\phi^*]$.
\end{proof}
%
\begin{corollary} \label{cor:orthogonal:proj}
 $\phi \to \mathcal{H}_\phi\/$ is an orthogonal projection with respect
 to the positive semi-definite quadratic form associated with
 $\mathcal{F}\/$ --- namely: the one used by Equation~\eqref{eq:fun:proof}.
\end{corollary}
%
The reason for the name ``stability functional'', and the relevance of
Theorem~\ref{thm:minimize:seminorm}, will become clear later in
\S~\ref{subsec:adv:funspaces}, where the inequality
$\mathcal{F}[\mathcal{H}_{\phi}] \le \mathcal{F}[\phi]\/$ is shown to
play a crucial role in ensuring the stability of superconsistent jet
schemes. Theorem~\ref{thm:minimize:seminorm} states that the Hermite
interpolant is within the class of the least oscillatory functions that
matches the data --- where ``oscillatory'' is measured by the functional
$\mathcal{F}\/$.
\begin{lemma} \label{lem:minimize:seminorm}
 The equality in \eqref{eq:fun:proof} applies.
\end{lemma}
%
\begin{proof}
 The proof is by induction over $p\/$. Without loss of generality,
 assume $Q = [0\/,\,1]^p\/$ is the unit $p$-cube. For $p=1\/$ the result
 holds, since $k+1\/$ integrations by parts can be used to obtain:
 \begin{equation*}
  I_1 = (-1)^{k+1}\,\int_0^1 \prn{\partial_1^{\,2k+2}\,
  \mathcal{H}_{\phi}(x_1)}\,\varphi(x_1)\ud{x_1} = 0\/,
 \end{equation*}
 where there are no boundary contributions because the data for
 $\varphi\/$ vanishes, and we have used that
 $\partial_1^{\,2k+2}\mathcal{H}_{\phi}(x) = 0\/$. Assume now that the
 result is true for $p-1>0\/$, and do $k+1\/$ integrations by parts
 over the variable $x_p\/$. This yields
 \begin{equation*}
  I_p = (-1)^{k+1} \int_Q \prn{\partial^{\,\vec{\beta}{\,}'}\,
  \partial_p^{2k+2}\,\mathcal{H}_{\phi}(\vec{x})}
  \prn{\partial^{\,\vec{\beta}{\,}'}\varphi(\vec{x})}\ud{\vec{x}} +
  \mbox{BTC} = \mbox{BTC}\/,
 \end{equation*}
 where BTC stands for ``Boundary Terms Contributions'',
 $\vec{\beta}{\,}' = (k+1,\dots,k+1)\/$ is a $(p-1)$-vector, and we have
 used that $\partial_p^{\,2k+2}\mathcal{H}_{\phi}(x) = 0\/$. Furthermore
 the BTC are a sum over terms of the form
 \begin{equation*}
  I_{p-1\/,\,j\/,\,q} = \int_{Q_q} \prn{\partial^{\,\vec{\beta}{\,}'}
  \partial_p^{k+j+1}\,\mathcal{H}_{\phi}(\vec{x})}
  \prn{\partial^{\,\vec{\beta}{\,}'} \partial_p^{k-j}\,
  \varphi(\vec{x})}\ud{\vec{x}}\/,\quad 0 \leq j \leq k\/,
  \quad 0 \leq q \leq 1\/,
 \end{equation*}
 where $Q_q\/$ stands for the $(p-1)$-cube obtained from $Q\/$ by setting
 either $x_p = 0\/$ ($q=0\/$) or $x_p = 1\/$ ($q=1\/$). Now: the data
 for $\varphi\/$ is, precisely, the union of the data for
 $\{\partial_p^{k-j}\,\varphi\}_{j=0}^k\/$ in both $Q_0\/$ and $Q_1\/$.
 Furthermore $\partial_p^{k+j+1}\,\mathcal{H}_{\phi}\/$, restricted to
 either $Q_0\/$ or $Q_1\/$, is a $(p-1)$-$n$ polynomial. Hence
 we can use the induction hypothesis to conclude that
 $I_{p-1\/,\,j\/,\,q} = 0\/$, for all choices of $j\/$ and $q\/$. Thus
 $I_p = \mbox{BTC} = 0\/$, which concludes the inductive proof.
\end{proof}

\subsection{Total and Partial $k$-Jets} \label{subsec:total:partial:jets}
%
The \emph{jet of a smooth function} $\phi:\Omega\to\mathbb{R}\/$ is the
collection of all the derivatives $\partial^{\,\vec{\alpha}}\phi\/$, where
$\vec{\alpha}\in \mathbb{N}_0^p\/$. The interpolants defined in
\S~\ref{subsec:global_interpolant} are based on parts of the jet of a
function, evaluated at grid points. Next we introduce two different notions
characterizing parts of jet.
%
\begin{definition} \label{def:total:k-jet}
 The \emph{total $k$-jet} is the collection of all the derivatives
 $\partial^{\,\vec{\alpha}}\phi\/$, where $\abs{\vec{\alpha}} \le k\/$.
\end{definition}
%
\begin{definition} \label{def:partial:k-jet}
 The \emph{partial $k$-jet} is the collection of all derivatives
 $\partial^{\,\vec{\alpha}}\phi\/$, where $\vec{\alpha}\in\{0,\dots,k\}^p\/$.
\end{definition}
%
Thus the total $k$-jet contains all the derivatives up to the order $k\/$,
while the partial $k\/$-jet contains all the derivatives for which the
partial derivatives with respect to each variable are at most order $k\/$.
For any given $k\ge 0$, the total $k$-jet is contained within the partial
$k$-jet (strictly if $k > 0\/$), while the partial $k$-jet is a subset of
the total $n$-jet, with $n=p\/k\/$.

\subsection{Projections in Function Spaces}
\label{subsec:projections_function_spaces}
%
The aim  of  this subsection is to use the global interpolant introduced
in \S~\ref{subsec:global_interpolant} to define projections of functions
--- which are needed for the projection step in the general class of
advect--and--project methods specified in \S~\ref{subsec:advect_project}.
This requires the introduction of appropriate spaces where the projections
operate. Further, because of the advection that occurs between projection
steps in the \S~\ref{subsec:advect_project} methods, it is necessary that
these spaces be invariant under diffeomorphisms. Unfortunately, the
coordinate bias (see Remark~\ref{rem:Ck:only}) that the generalized
Hermite interpolants exhibit causes difficulties with this last
requirement, as explained next.

The functions that we are interested in projecting are smooth advections,
over one time step, of some global interpolant $\mathcal{H}\/$. Let
$\phiSP\/$ be such a function. From Remark~\ref{rem:Ck:only}, it follows
that $\phiSP\/$ is $C^k\/$ and piece-wise smooth, with singularity
hypersurfaces given by the advection of the grid hyperplanes. Inside each
of the regions that result from advecting a single grid cell, $\phiSP\/$
is smooth all the way to the boundary of the region. Furthermore, at a
singularity hypersurface:
\begin{itemize}
 \item[(i)]
 Partial derivatives of $\phiSP\/$ involving differentiation normal
 to the hypersurface of order $k\/$ or lower are defined and continuous.
 \item[(ii)]
 Partial derivatives of $\phiSP\/$ involving differentiation normal to
 the hypersurface of order $k+1\/$ are defined, but (generally) have a
 simple jump discontinuity.
 \item[(iii)]
 Partial derivatives of $\phiSP\/$ involving differentiation normal to
 the hypersurface of order higher than $k+1\/$ are (generally) not
 defined as functions. However, they are defined on each side, and
 are continuous up to the hypersurface.
\end{itemize}
Imagine now a situation where a grid point is on one of the singularity
hypersurfaces. Then the interpolant $\mathcal{H}_{\phiSP}\/$, as given by
Definition~\ref{def:H_phi}, may not exist --- because one, or more, of the
partial derivatives of $\phiSP\/$ needed at the grid point does not exist.
This is a serious problem for the advect--and--project strategy advocated
in \S~\ref{subsec:advect_project}. This leads us to the considerations
below, and a recasting of Definition~\ref{def:H_phi}
(i.e.: Definition~\ref{def:H_phi:Num}) that resolves the problem.

We are interested in solutions to Equation~\eqref{eq:linear_advection}
that are are sufficiently smooth.
Let then $\phiSP\/$ be an approximation to such a solution. In this
case, from Theorem~\ref{thm:p-n_polynomial_accuracy}, we should add to
(i -- iii) above the following
\begin{itemize}
 \item[(iv)]
 The jumps in any partial derivative of $\phiSP\/$ (across a singularity
 hypersurface) cannot be larger than $O(h^{n+1-s})\/$, where $n=2\/k+1\/$
 and $s\/$ is the order of the derivative.
\end{itemize}
Then, from Lemma~\ref{lem:p-n_polynomial_accuracy_data}, we see that these
jumps are below the numerical resolution --- hence, essentially, not
present from the numerical point of view. This motivates the following
generalization of Definition~\ref{def:H_phi}:
%
\begin{definition} \label{def:H_phi:Num}
 For functions $\phiSP\/$ that satisfy the restrictions in items~(i -- iv)
 above, define the global interpolant $\mathcal{H}_\phiSP\/$ as in
 Definition~\ref{def:H_phi}, with one exception:
 whenever a partial derivative
 $\phiSP_{\vec{\alpha}}^{\vec{m}} = \partial^{\,\vec{\alpha}}
 \phiSP(\vec{x}_{\vec{m}})\/$ which is needed for the data at a grid point
 does not exist (because the grid point is on a singularity
 hypersurface), supply a value using the formula
 \begin{equation} \label{eqn:partials:by:limsup}
  \partial^{\vec{\alpha}} \phiSP(\vec{x}_{\vec{m}}) = \frac{1}{2}\,\prn{
   \limsup_{\vec{x} \to \vec{x}_{\vec{m}}} \partial^{\vec{\alpha}}
   \phiSP(\vec{x}) + \liminf_{\vec{x} \to \vec{x}_{\vec{m}}}
   \partial^{\vec{\alpha}} \phiSP(\vec{x})}\/.
 \end{equation}
\end{definition}
%
\begin{remark} \label{rem:justify:limsup:def}
 The value in equation \eqref{eqn:partials:by:limsup} for
 $\partial^{\vec{\alpha}} \phiSP(\vec{x}_{\vec{m}})\/$ ignores the
 ``distribution component'' of
 $\partial^{\vec{\alpha}} \phiSP(\vec{x}_{\vec{m}})\/$ --- i.e. Dirac's
 delta functions and derivatives, with support on the singularity
 hypersurface. These components originate purely because of
 approximation errors: small mismatches in the derivatives of the
 Hermite interpolants across the grid hyperplanes. Thus we expect
 Equation~\eqref{eqn:partials:by:limsup} to provide an accurate
 approximation to the corresponding partial derivative of the smooth
 function that $\phiSP\/$ approximates. Note that, while we are unable
 to supply a rigorous justification for this argument, the numerical
 experiments that we have conducted provide a validation ---
 see \S~\ref{sec:numerical_results}.
\end{remark}
%
\begin{remark} \label{rem:alternative:def:H_phi:Num}
 An alternative to Definition~\ref{def:H_phi:Num}, that also allows the
 construction of global interpolants from advected global interpolants,
 is introduced in Definition~\ref{def:projection_total} below.
\end{remark}
%
Motivated by the discussion above, we now consider two types of
projections (and corresponding function spaces). Note that in the
theoretical formulation that follows below, the ``small jumps in the
partial derivatives'' property is not built into the definitions of the
spaces. This is not untypical for numerical methods. For instance, finite
difference discretizations are only meaningful for sufficiently smooth
functions. However, the finite differences themselves are defined for
any function that lives on the grid, even though the notion of
smoothness does not make sense for such grid-functions.
%
\begin{definition} \label{def:Snu:space}
 Let $S^\nu\/$ denote the space of of all the $L^\infty\/$ functions
 $\phiSP:\Omega\to\mathbb{R}\/$ which are $\nu$-times differentiable
 a.e., with derivatives in $L^\infty\/$.

Here differentiable is meant in the classical sense:
 $\phiSP(\vec{x}+\vec{u}) = \phiSP(\vec{x}) +
 \prn{D\,\phiSP}[\vec{u}] +
 \frac{1}{2}\/\prn{D^2\,\phiSP}[\vec{u}\/,\vec{u}] +\,\dots\,+
 \frac{1}{\nu!}\/\prn{D^\nu\,\phiSP}[\vec{u}\/,\dots\/,\,\vec{u}] +
 o\prn{\|\vec{u}\|^\nu}\/$
near a point where the function is differentiable,
where $\prn{D^\mu\,\phiSP}\/$, $1 \leq \mu \leq \nu\/$,
is a $\mu$-linear form.
\end{definition}
%
\begin{definition} \label{def:Sk+:space}
 Let $S^{k\/,\,+}\/$ be defined by  $S^{k\/,\,+} = S^\nu \cap C^k\/$,
 with $\nu = p\/k\/$.
\end{definition}
%
\begin{definition} \label{def:define:pointwise}
 For any function $\phiSP \in S^\nu\/$, and any
 $\vec{\alpha} \in \mathbb{N}^p_0\/$ with $|\vec{\alpha}| \leq \nu\/$,
 define $\partial^{\vec{\alpha}} \phiSP\/$ at every point
 $\vec{x}_{\vec{m}}\/$ in the rectangular grid, via
 \begin{equation} \label{eqn:partials:by:esslimsup}
  \partial^{\vec{\alpha}} \phiSP(\vec{x}_{\vec{m}}) = \frac{1}{2}\,\prn{
  \mathrm{ess} \limsup_{\vec{x} \to \vec{x}_{\vec{m}}}
  \partial^{\vec{\alpha}} \phiSP(\vec{x}) +
  \mathrm{ess} \liminf_{\vec{x} \to \vec{x}_{\vec{m}}}
  \partial^{\vec{\alpha}} \phiSP(\vec{x})}\/.
 \end{equation}
 where $\mathrm{ess}\limsup\/$ and  $\mathrm{ess}\liminf\/$ denote the
 essential upper and lower limits~\cite{Arutyunov2000}, respectively.
 Notice that this last formula differs from
 \eqref{eqn:partials:by:limsup} by the use of essential limits only.

 The definition in Equation~\eqref{eqn:partials:by:esslimsup} is not the
 only available option. One could use the upper limit only, or the lower
 limit only, or some intermediate value between these two --- since
 numerically the various alternatives lead to differences that are of the
 order of the truncation errors. In particular, a single sided evaluation
 is more efficient, so this is what we use in our numerical implementations.
 In what follows we assume that a choice has been made, e.g.: the one
 given by Equation~\eqref{eqn:partials:by:esslimsup}, or the one given
 in Remark~\ref{rem:analytic:update:gridlines}.
\end{definition}
%
\begin{definition} \label{def:projection_partial}
 In $S^{k\/,\,+}\/$, define the projection
 $P^+_k:S^{k\/,\,+}\to S^{k\/,\,+}\/$ as the application of the interpolant
 defined in \S~\ref{subsec:global_interpolant},
 i.e.~$P^+_k\phiSP = \mathcal{H}_{\phiSP}\/$.
\end{definition}
%
\begin{definition} \label{def:projection_total}
 In $S^k\/$, define the projection $P_k:S^k\to S^k\/$ by the following
 steps.
 \begin{enumerate}[(1)]
  \item
  Given a function $\phiSP\in S^k\/$, evaluate the total $k$-jet at the
  grid points $\vec{x}_{\vec{m}}\/$.
  \item
  Approximate the remaining derivatives in the partial $k$-jet by a
  process like the one described in \S~\ref{subsec:optimally:local}. Notice
  that, if the approximation of the derivatives of order $\ell>k\/$ is done
  with $O(h^{n+1-\ell})\/$ errors, then (by
  Lemma~\ref{lem:p-n_polynomial_accuracy_data}) the full accuracy of the
  projection is preserved.
  \item
  Define $P_k\phiSP\/$ as the global interpolant
  (see \S~\ref{subsec:global_interpolant}) based on this approximate
  partial $k$-jet.
 \end{enumerate}
\end{definition}
%

\subsection{Construction of the Partial $k$-jet from the Total $k$-jet.}
\label{subsec:optimally:local}
%
For the advect--and--project methods introduced in
\S~\ref{subsec:advect_project}, the advection of the solution's derivatives
is a significant part of the computational cost. Furthermore, the
higher the derivative, the higher the cost. For example, when using
$P_k^+\/$ with $k=2\/$ and the approach in \S~\ref{subsubsec:epsilon:FD},
the cost of obtaining the $|\vec{\alpha}| > k\/$ derivatives in the partial
$k$-jet is (at the lowest) about three times that of obtaining the total
$k$-jet. Further: the cost ratio grows, roughly, linearly with $k\/$.
Thus it is tempting to design more efficient approaches that are based on
advecting the total $k$-jet only, and then recovering the higher
derivatives in the partial $k$-jet by a finite difference approximation of
the total $k$-jet data at grid points. We call such approaches
\emph{grid-based finite differences}, in contrast to the
\emph{$\varepsilon$-finite differences} introduced in
\S~\ref{subsubsec:epsilon:FD}.

Unfortunately, jet schemes based on grid-based finite difference
reconstructions of the higher derivatives in the partial $k$-jet have
one major drawback: the nice stabilizing properties that the Hermite
interpolants possess (see Theorem~\ref{thm:minimize:seminorm} and
\S~\ref{subsec:adv:funspaces}) are lost. Hence the
stability of these schemes is not ensured. Nevertheless, we explore this
idea numerically (see \S~\ref{subsec:numerics_stability} for some
examples). As expected, in the absence of a theoretical underpinning that
would allow us to distinguish stable schemes from their unstable
counterparts, many approaches that are based on grid-based finite
differences turn out to yield unstable schemes. However
%
\begin{itemize}
 \item
 In the case $k=1\/$, a stable scheme using a grid-based finite difference
 reconstruction of the partial $k$-jet, can be given. It is described in
 Example~\ref{ex:optimal:local:approx}. The particular reconstruction used
 has some interesting properties, which motivate
 Definition~\ref{def:optimally:local} below.
 \item
 In unstable versions of schemes that are based on grid-based
 reconstructions of higher derivatives, the instabilities are, in general,
 observed to be very weak: even for fairly small grid sizes, the growth
 rate of the instabilities is rather small --- an example of this
 phenomenon is shown in \S~\ref{subsec:numerics_stability}. Thus it may be
 possible to stabilize these schemes (e.g.\ by some form of jet scheme
 artificial viscosity) without seriously diminishing their accuracy.
\end{itemize}
%
\begin{definition} \label{def:optimally:local}
 We say that a projection is \emph{optimally local} if, at each grid node,
 the recovered data (e.g. derivatives with $|\vec{\alpha}| > k\/$ in the
 partial $k$-jet) depends solely on the known data (e.g. the total
 $k$-jet) at the vertices of a single grid cell.
\end{definition}
%
Of course, the projections $P_k^+\/$ in
Definition~\ref{def:projection_partial} are by default optimally local
(since there is no data to be recovered). However, as pointed out above,
the idea of optimally local projections that are based on the total
$k$-jet is that they produce more efficient jet schemes than using
$P_k^+\/$. In Example~\ref{ex:optimal:local:approx} below we present an
optimally local projection, which corresponds to $P_1\/$ in
Definition~\ref{def:projection_total}.
%
\begin{remark} \label{rem:why:oplocal}
 An important question is: \emph{why is optimal locality desirable?} The
 reason is that optimally local projections do not involve communication
 between cells. This has, at least, three advantageous consequences.
 First, near boundaries, a local formulation simplifies the enforcement
 of boundary conditions (e.g.\ see \S~\ref{subsec:Boundary:Conditions}).
 Second, in situations where adaptive grids are used, a local formulation
 can make the implementation considerably simpler than it would be for
 non-local alternatives. And third, locality is a desirable property for
 parallel implementations since communication boundaries between
 processors are reduced to lines of single cells.
\end{remark}
%
\begin{example} \label{ex:optimal:local:approx}
 In two space dimensions, we can define an optimally local version of the
 projection $P_1\/$, as follows below. In this projection we presume that
 the total $1$-jet is given at each grid point. To simplify the description,
 we work in the cell $Q = [0\/,\,h]\times[0\/,\,h]\/$, with
 $\vec{q} \in \{0\/,\,1\}^2\/$ corresponding to the vertex
 $\vec{x}_{\vec{q}} = \vec{q}\,h\/$. Thus $\phiSP^{\vec{q}}\/$ indicates the
 value $\phiSP(\vec{x}_{\vec{q}})\/$. We also use the convention that when
 $q_i = \2$, the quantity is defined or evaluated at $x_i = \frac{h}{2}$,
 for $i\in\{1\/,\,2\}$.

 \noindent
 Define $P_1\/$ as follows (this is the projection introduced, and
 implemented, in~\cite{NaveRosalesSeibold2010}). To construct the bi-cubic
 interpolant, at each vertex of $Q\/$ the derivative $\phiSP_{xy}\/$ must
 be approximated with second order accuracy, using the given data at the
 vertices. To do so:
 (i)   Use centered differences to  write second order approximations to
       $\phiSP_{xy}\/$ at the midpoints of each of the edges of $Q\/$.
       For example: $\phiSP_{xy}^{(0,\2)} =
       \frac{1}{h}\,\prn{\phiSP_x^{(0,1)}-\phiSP_x^{(0,0)}}\/$
       and $\phiSP_{xy}^{(\2,0)} =
       \frac{1}{h}\,\prn{\phiSP_y^{(1,0)}-\phiSP_y^{(0,0)}}\/$.
 (ii)  Use the four values obtained in (i) to define a bi-linear
       (relative to the rotated coordinate system $x+y\/$ and $x-y\/$)
       approximation to $\phiSP_{xy}\/$.
 (iii) Evaluate the bi-linear approximation obtained in (ii) at the
       vertices of $Q\/$.
 Three dimensional versions of these formulas are straightforward, albeit
 somewhat cumbersome.
\end{example}
%
We do not know whether stable analogs of $P_1\/$ exist for $P_k\/$ with
$k>1\/$. It is plausible that a construction similar to the one given in
Example~\ref{ex:optimal:local:approx} (i.e.\ based on using lower order
Hermite interpolants in rotated coordinate systems) may yield the desired
stability properties. This is the subject of current research.
%
\begin{remark} \label{rem:oplocal:noCk}
 The optimally local projection $P_1\/$ with ``single cell
 approximations'' uses, at each grid point, values for $\phiSP_{xy}\/$
 that are different for each of the adjoining cells. This means that
 (generally) $P_1 \phiSP \notin C^1\/$ --- though $P_1 \phiSP \in C^0\/$
 always. However, for sufficiently smooth functions $\phiSP\/$, the
 discontinuities in the derivatives of $P_1\,\phiSP\/$ across the grid
 lines cannot involve jumps larger than $O(h^{4-s})\/$, where $s\/$ is
 the order of the derivative. This is a small variation of the
 situation discussed from the beginning of
 \S~\ref{subsec:projections_function_spaces} through
 Remark~\ref{rem:justify:limsup:def}. The same arguments made there
 apply here.
\end{remark}
%
\begin{remark} \label{rem:oplocal:average}
 The lack of smoothness of the optimally local projection $P_1\/$ (i.e.
 Remark~\ref{rem:oplocal:noCk}) could be eliminated as follows. At each
 grid node consider all the possible values obtained for $\phiSP_{xy}\/$
 by cell based approximations (one per cell adjoining the node). Then
 assign to $\phiSP_{xy}\/$ the average of these values. This eliminates
 the multiple values, and yields an interpolant $P_1\/\phiSP\/$ that
 belongs to $C^1\/$.
 Formally this creates a dependence of the partial $1$-jet on the total
 $1$-jet that is not optimally local. However, all the operations done
 are solely cell-based (i.e.: the reconstruction of the partial $1$-jet)
 or vertex based (i.e.: the averaging). Hence, such an approach can
 easily be applied at boundaries or in adaptive meshes.
\end{remark}

\section{Advection and Update in Time} \label{sec:Advect:Update:Time}
%
The characteristic form of Equation~\eqref{eq:linear_advection} is
\begin{align}
 \frac{d\,\vec{x}}{d\/t} & = \vec{v}(\vec{x}\/,\,t)\/,
                                 \label{eq:characteristic_system_x} \\
 \frac{d\,\phi}{d\/t}    &= 0\/. \label{eq:characteristic_system_phi}
\end{align}
Let $\vec{X}(\vec{x}\/,\,\tau\/,\,t)\/$ denote the solution of the ordinary
differential equation~\eqref{eq:characteristic_system_x} at time $t\/$,
when starting with initial conditions $\vec{x}\/$ at time $t = \tau\/$.
Hence $\vec{X}\/$ is defined by
\begin{equation*}
 \pd{}{t}\vec{X}(\vec{x}\/,\,\tau\/,\,t) =
 \vec{v}(\vec{X}(\vec{x}\/,\,\tau\/,\,t)\/,\,t)\/, \quad \mbox{with}
 \quad \vec{X}(\vec{x}\/,\,\tau\/,\,\tau) = \vec{x}\/.
\end{equation*}
Due to \eqref{eq:characteristic_system_phi}, the solution of the partial
differential equation~\eqref{eq:linear_advection} satisfies
\begin{equation*}
 \phi(\vec{x}\/,\,t+\Delta t) =
 \phi(\vec{X}(\vec{x}\/,\,t+\Delta t\/,\,t)\/,\,t)\/.
\end{equation*}
That is: the solution at time $t+\Delta t\/$ and position $\vec{x}\/$ is
found by tracking the corresponding characteristic curve, given by
\eqref{eq:characteristic_system_x}, backwards to time $t\/$, and evaluating
the solution at time $t\/$ there.
Introduce now the solution operator $S_{\tau\/,\,t}\/$,
which maps the solution at time $t\/$ to the solution at time $\tau\/$.
Then $S_{\tau\/,\,t}\/$ acts on a function $g(\vec{x})\/$ as follows
\begin{equation} \label{eq:advection_operator}
 (S_{\tau\/,\,t}\,g)(\vec{x}) = g(\vec{X}(\vec{x}\/,\,\tau\/,\,t))\/.
\end{equation}
Hence, the solution of \eqref{eq:linear_advection} satisfies
\begin{equation*}
 \phi(\vec{x}\/,\,t+\Delta t) = S_{t+\Delta t\/,\,t}\,\phi(\vec{x}\/,\,t)\/.
\end{equation*}
In particular, the solution at time $t=n\,\Delta t\/$ can be obtained by
applying successive advection steps to the initial conditions
\begin{equation*}
 \phi(\vec{x}\/,\,t) = S_{t\/,\,t-\Delta t} \circ \dots \circ
 S_{2\Delta t\/,\, \Delta t} \circ S_{\Delta t\/,\,0}\,\phiIC(\vec{x})\/.
\end{equation*}

\subsection{Approximate Advection} \label{subsec:Approx:Advec}
%
Let $\vec{\mathcal{X}}\/$ be an approximation to the exact solution
$\vec{X}\/$, as arising from a numerical ODE solver --- e.g. a high order
Runge-Kutta method. By analogy to the true advection operator
\eqref{eq:advection_operator}, introduce an approximate advection
operator, defined as acting on a function $g(\vec{x})\/$ as follows
\begin{equation} \label{eq:advection_operator_approx}
 (A_{\tau\/,\,t}\,g)(\vec{x}) =
 g(\vec{\mathcal{X}}(\vec{x}\/,\,\tau\/,\,t))\/.
\end{equation}
The approximate characteristic curves, given by $\vec{\mathcal{X}}\/$,
motivate the following definition.
\begin{definition} \label{def:foot}
 The point $\xfoot = \vec{\mathcal{X}}(\vec{x}\/,\,t+\Delta t\/,\,t)\/$
 is called the foot of the (approximate) characteristic through
 $\vec{x}\/$ at time $t+\Delta\/t\/$.
 Note that $\xfoot = \xfoot(\vec{x}\/,\,t\/,\,\Delta t)\/$, but we
 do not display these dependencies when obvious.
\end{definition}
Let now $\vec{\mathcal{X}}(\vec{x}\/,\,t+\Delta t\/,\,t)\/$ represent
a single ODE solver step for \eqref{eq:characteristic_system_x}, from
$t+\Delta t\/$ to $t\/$, starting from $\vec{x}\/$ --- i.e. let
$\Delta t\/$ be the ODE solver time step. Then the successive
application of approximate advection operators
\begin{equation} \label{eq:evolution_to_zero_approx}
 A_{t\/,\,t-\Delta t} \circ \dots \circ
 A_{2\Delta t\/,\, \Delta t} \circ A_{\Delta t\/,\,0}\,\phiIC(\vec{x})
\end{equation}
yields an approximation to the solution of \eqref{eq:linear_advection} at
time $t=n\Delta t\/$. In principle this formula provides a way to
(approximately) solve \eqref{eq:characteristic_system_x} by taking (for
each point $\vec{x}\/$ at which the solution is desired at time
$t=n\,\Delta t\/$) $n\/$ ODE solver steps from time $t\/$ back to time
$0\/$, and then
evaluating the initial conditions at the resulting position. However, as
described in \S~\ref{sec:introduction}, we are interested in approaches
that allow access to the solution at each time step (represented on a
numerical grid by a finite amount of data), and then advance it forward
to the next time step. Hence the method provided by expression
\eqref{eq:evolution_to_zero_approx} is not adequate.

\subsection{Advect--and--Project Approach} \label{subsec:advect_project}
%
In this approach, an appropriate projection is applied at the end of every
time step, after the advection. The projection allows the representation
of the solution, at each of a discrete set of times, with a finite amount
of data.
Here, we consider projections based on function and derivative evaluations
at the grid points, as in \S~\ref{subsec:projections_function_spaces}.
Thus, let $P\/$ denote a projection operator --- e.g.\ see Definitions
\ref{def:projection_partial}, \ref{def:projection_total}, or
\ref{def:optimally:local}.
Then the approximate solution method is defined by
\begin{equation}
\label{eq:approxiate_solution}
 \phi_{\text{approx}}(\vec{x},t) =
 \prn{P \circ A_{t\/,\,t-\Delta t}} \circ \dots \circ
 \prn{P \circ A_{2\Delta t\/,\,\Delta t}} \circ
 \prn{P \circ A_{\Delta t\/,\,0}}\,\phiIC(\vec{x})\/.
\end{equation}
Namely, to advance from time $t\/$ to $t+\Delta t\/$, apply
$P \circ A_{t+\Delta t\/,\,t}\/$ to the (approximate) solution at time $t\/$.
The key simplification introduced by adding the projection step is that: in
order to define the (approximate) solution at time $t+\Delta t\/$, only the
data at the grid points is needed. At this point, all that is missing to
make this approach into a computational scheme, is a method to obtain the
grid point data at time $t+\Delta t\/$ from the (approximate) solution at
time $t\/$. This is done in \S~\ref{subsec:evolution:k-jet}.

As shown in \S~\ref{subsec:generalized_hermite_interpolation}, the
$P_k\/$ projection is an $O(h^{n+1})\/$ accurate approximation for
sufficiently smooth functions, where $n=2\/k+1\/$ and
$h = \max_{\/i} \Delta x_i\/$.
Thus, with $\Delta t \propto h\/$ the use of a locally
$(n+1)^{st}\/$ order time stepping scheme ensures that the full accuracy
is preserved. As usual, the global error is in general one order less
accurate, since $O(1/h)\/$ time steps are required to reach a given final
time. Hence, a $k$-jet scheme can be up to $n^{th}\/$ order accurate.

\subsection{Evolution of the $k$-Jet} \label{subsec:evolution:k-jet}
%
The projections defined in \S~\ref{subsec:projections_function_spaces}
require knowledge of parts of the jet of the (approximate) solution at the
grid points. Specifically, $P^+_k\/$ (see
Definition~\ref{def:projection_partial}) requires the partial $k$-jet,
and $P_k\/$ (see Definition~\ref{def:projection_total}) requires the
total $k$-jet.

A natural way to find the jet at time $t+\Delta t\/$ is to consider the
approximate advection operator \eqref{eq:advection_operator_approx}. Since
it defines an approximate solution everywhere, it also defines the
solution's spacial derivatives. Thus, by differentiating
\eqref{eq:advection_operator_approx}, update rules for all the elements of
the $k$-jet can be systematically inherited from the numerical scheme for
the characteristics.
%
\begin{definition} \label{def:superconsistent}
 Jet schemes for which the update rule for the whole (total or partial)
 $k$-jet is derived from one single approximate advection scheme are
 called \emph{superconsistent}.
\end{definition}
%
\begin{remark} \label{rem:nonsuperconsistent}
 Non-superconsistent jet schemes can be constructed, by first writing an
 evolution equation along the characteristics for each of the relevant
 derivatives --- by differentiating Equation~\eqref{eq:linear_advection},
 and then applying some approximation scheme to each equation.
 On the one hand, these schemes can be less costly than superconsistent
 ones, because an $\ell^{th}$ order derivative only needs
 an accuracy that is $\ell\/$ orders below that of the function value ---
 see Lemma~\ref{lem:p-n_polynomial_accuracy_data}. On the other hand, the
 benefits of an interpretation in function spaces are lost, such as the
 optimal coherence between all the entries in the $k$-jet, and the
 stability arguments of \S~\ref{subsec:adv:funspaces}.
\end{remark}
%
Notice that a superconsistent scheme has a very special property: even
though it is a fully discrete process, it updates the solution in time
in a way that is equivalent to carrying out the process given by
Equation~\eqref{eq:approxiate_solution} in the function space where the
projection is defined. Below we present two possible ways to find the
update rule for the $k$-jet: analytical differentiation and
$\varepsilon$-finite differences. Up to a small approximation error,
both approaches are equivalent. However, they can make a difference
with the ease of implementation.

\subsubsection{Analytical differentiation}
\label{subsubsec:Analytical_Diff}
%
Let $\mathcal{H}\/$ denote the approximate solution at time $t\/$. Since
each time step ends with a projection step, $\mathcal{H}\/$ is a piecewise
$p$-$n$ polynomial, defined by the data at time $t\/$. Using the short
notation
$\vec{\mathcal{X}} = \vec{\mathcal{X}}(\vec{x}\/,\,t+\Delta t\/,\,t)\/$,
update formulas for the derivatives (here up to second order) are given by:
\begin{align*}
 A_{t+\Delta t\/,\,t}\,\phi(\vec{x}\/,\,t)
 & = \mathcal{H}(\vec{\mathcal{X}}\/,\,t)\/,
 \\
 \frac{\partial}{\partial x_i}
 A_{t+\Delta t\/,\,t}\,\phi(\vec{x}\/,\,t)
 & = \frac{\partial\vec{\mathcal{X}}}{\partial x_i}
 \cdot D\,\mathcal{H}(\vec{\mathcal{X}}\/,\,t)\/,
 \\
 \frac{\partial^2}{\partial x_i\partial x_j}
 A_{t+\Delta t\/,\,t}\,\phi(\vec{x}\/,\,t)
 & = \frac{\partial^2\vec{\mathcal{X}}}{\partial x_i\,\partial x_j}
 \cdot D\,\mathcal{H}(\vec{\mathcal{X}}\/,\,t) +
 \prn{\frac{\partial\vec{\mathcal{X}}}{\partial x_i}}^T
 \cdot D^2\,\mathcal{H}(\vec{\mathcal{X}}\/,\,t)
 \cdot \frac{\partial\vec{\mathcal{X}}}{\partial x_j}\/,
\end{align*}
where $D\mathcal{H}\/$ is the Jacobian and $D^2\mathcal{H}\/$ the Hessian
of $\mathcal{H}\/$. It should be clear how to continue this pattern for
higher derivatives. Since $\mathcal{H}\/$ is a $p$-$n$ polynomial, its
derivatives are easy to compute analytically. The partial
derivatives of $\vec{\mathcal{X}}\/$ follow from the ODE solver formulas,
as the following example (using a Runge-Kutta scheme) illustrates.
%
\begin{example} \label{ex:ShuOsher:analytic}
 When tracking the total $1$-jet (this would be called a
 \emph{gradient-augmented scheme}, see \cite{NaveRosalesSeibold2010}), the
 Hermite interpolant is fourth order accurate. Hence, in order to achieve
 full accuracy when scaling $\Delta t\propto h\/$, the advection operator
 should be approximated with a locally fourth order accurate time-stepping
 scheme. An example is the Shu-Osher scheme \cite{ShuOsher1988},
 which in superconsistent form looks as follows.
 \begin{equation*}
  \begin{split}
   \vec{x}_1
   & = \vec{x} - \Delta t \; \vec{v}\,(\vec{x}\/,\,t+\Delta t)\/,
   \\
   \nabla\,\vec{x}_1
   & = I - \Delta t \; \nabla \, \vec{v}\,(\vec{x}\/,\,t+\Delta t)\/,
   \\
   \vec{x}_2
   & = \tfrac{3}{4}\,\vec{x} + \tfrac{1}{4}\,\vec{x}_1 -
   \tfrac{1}{4}\,\Delta t\;\vec{v}\,(\vec{x}_1\/,\,t)\/,
   \\
   \nabla\vec{x}_2
   & = \tfrac{3}{4}\,I + \tfrac{1}{4}\,\nabla\vec{x}_1 -
   \tfrac{1}{4}\,\Delta t\;\nabla \,\vec{x}_1 \cdot \nabla\,
   \vec{v}\,(\vec{x}_1\/,\,t)\/,
   \\
   \xfoot
   & = \tfrac{1}{3}\,\vec{x} + \tfrac{2}{3}\,\vec{x}_2 - \tfrac{2}{3}\,
   \Delta t\;\vec{v}\,(\vec{x}_2\/,\,t+\tfrac{1}{2}\,\Delta t)\/,
   \\
   \nabla\,\xfoot
   & = \tfrac{1}{3}\,I + \tfrac{2}{3}\,\nabla\,\vec{x}_2 -
   \tfrac{2}{3}\,\Delta t\;\nabla\,\vec{x}_2 \cdot
   \nabla\,\vec{v}\,(\vec{x}_2\/,\,t+\tfrac{1}{2}\,\Delta t)\/,
   \\
   \phi(\vec{x}\/,\,t+\Delta t)
   & = \mathcal{H}(\xfoot\/,\,t)\/,
   \\
   (\nabla\phi)(\vec{x}\/,\,t+\Delta t)
   & = \nabla\,\xfoot \cdot \nabla\,\mathcal{H}(\xfoot\/,\,t)\/.
  \end{split}
 \end{equation*}
\end{example}
%
\begin{remark} \label{rem:analytic:update:gridlines}
 When $\xfoot\/$ is on a grid hyperplane,
 Equation~\eqref{eqn:partials:by:esslimsup} in
 Definition~\ref{def:define:pointwise} would require several
 ODE solves for derivatives that are discontinuous across the hyperplane
 (one for each of the local Hermite interpolants used in the adjacent grid
 cells). This would augment the computational cost without any increase in
 accuracy. Hence, in our implementations we do only one update, as follows:
 (i) A convention is selected that assigns the points on the grid
 hyperplanes as belonging to (a single) one of the cells that they are
 adjacent to.
 (ii) This convention then uniquely defines which Hermite interpolant
 should be used whenever $\xfoot\/$ is at a grid hyperplane.
\end{remark}
%
Finally, we point out that
\begin{itemize}
 \item
 The SSP property~\cite{GottliebShuTadmor2001} is not required for the
 tracking of the characteristics, since this is not a process in which TVD
 stability plays any role. The use of the Shu-Osher
 scheme~\cite{ShuOsher1988} in Example~\ref{ex:ShuOsher:analytic} is
 solely to illustrate the analytical differentiation process; any
 Runge-Kutta scheme of the appropriate order will do the job. In particular,
 when using $P_2^+\/$ a $5^{th}\/$ order Runge-Kutta scheme
 (e.g.\  the $5^{th}\/$ order component of the Cash-Karp
 method~\cite{CashKarp1990}) yields a globally $5^{th}\/$ order jet scheme.
 \item
 Jet schemes that result from using $P_0^+\/$ ($p$-linear interpolants) do
 not need any updates of derivatives, since they are based on function
 values only. In the one dimensional case, the CIR
 method~\cite{CourantIsaacsonRees1952} by Courant, Isaacson, and Rees is
 recovered. For these schemes a simple forward Euler ODE solver is
 sufficient.
 \item
 Using analytical differentiation, the update of a single partial
 derivative is almost as costly as updating all the partial derivatives of
 the same order. In particular, in order to update the partial $k$-jet all
 the derivatives of order less than or equal to $p\/k\/$ have to be
 updated --- even though only a few with order greater than $k\/$ are
 needed. Therefore, it is worthwhile to employ alternative ways to
 approximate the $|\vec{\alpha}|>k\/$ derivatives in the partial $k$-jet.
 One strategy is to use the grid-based finite differences described in
 \S~\ref{subsec:optimally:local}, assuming that a stable version can be
 found. Another strategy, which is essentially equivalent to analytical
 differentiation, is presented next.
\end{itemize}

\subsubsection{$\varepsilon$-finite differences} \label{subsubsec:epsilon:FD}
%
Here we present a way to update derivatives in a superconsistent fashion
that avoids the problem discussed in the third item at the end of
\S~\ref{subsubsec:Analytical_Diff}.
In $\varepsilon$-finite differences the definition of a superconsistent
scheme is applied by directly differentiating the approximately advected
function $A_{t+\Delta\/t\/,\,t}\,\phi(\vec{x}\/,\,t)\/$, or its derivatives.
This is done using finite difference formulas, with a separation
$\varepsilon \ll h\/$ chosen so that maximum accuracy is obtained. As
shown below, $\varepsilon$-finite differences can be either employed
singly, or in combination with analytic differentiation. In particular,
in order to update the partial $k$-jet, one could update the total
$k$-jet using analytical differentiation, and use $\varepsilon$-finite
differences to obtain the remaining partial derivatives. The approach is
best illustrated with a few examples.
%
\begin{example} \label{ex:epsilon:FD}
 Here we present examples of how to implement the advect--and--project
 approach given by Equation~\eqref{eq:approxiate_solution} in two
 dimensions, with $P = P_1^+\/$ or $P = P_2^+\/$
 (see Definition~\ref{def:projection_partial}), using a combination
 of analytical differentiation and $\varepsilon$-finite differences.
 The error analysis below estimates the difference between the
 derivatives obtained using $\varepsilon$-finite differences and the
 values required by superconsistency.
 This difference can be made much smaller than the
 numerical scheme's approximation errors, so that stability (see
 \S~\ref{subsec:adv:funspaces}) is preserved.
 In this analysis, $\delta > 0\/$ characterizes the accuracy of the
 floating point operations. Namely, $\delta\/$ is the smallest positive
 number such that $1+\delta \neq 1\/$ in floating point arithmetic.
 \begin{itemize}
  \item
  Advect--and--project using $P_1^+\/$. At time $t+\Delta t\/$ and at every
  grid point $(x_m\/,\,y_m)\/$, the values for
  $(\phi\/,\,D\/\phi\/,\,\phi_{x\/y})\/$ must be computed from the solution
  at time $t\/$. To do so, use the approximate advection solver
  $\vec{\mathcal{X}}\/$ to compute the approximately advected solution
  $\phi^{\vec{q}}\/$ at the four points
  $(x_m+q_1\,\varepsilon\/,\, y_m + q_2\,\varepsilon)\/$,
  where $\vec{q}\in\{-1\/,\,1\}^2$.
  From these values
  $\phi^{(1,1)}$, $\phi^{(-1,1)}$, $\phi^{(1,-1)}$, and $\phi^{(-1,-1)}$
  we obtain the desired derivatives at the grid point
  using the $O(\varepsilon^2)\/$ accurate stencils
  \begin{align*}
  \phi_{\phantom{xy}} &=
  \;\,\tfrac{1}{4}\,(\phi^{(1,1)}+\phi^{(-1,1)}
  +\phi^{(1,-1)}+\phi^{(-1,-1)})\;, \\
  \phi_{x\phantom{y}} &=
  \,\tfrac{1}{4\varepsilon}\,(\phi^{(1,1)}-\phi^{(-1,1)}
  +\phi^{(1,-1)}-\phi^{(-1,-1)})\;, \\
  \phi_{y\phantom{x}} &=
  \,\tfrac{1}{4\varepsilon}\,(\phi^{(1,1)}+\phi^{(-1,1)}
  -\phi^{(1,-1)}-\phi^{(-1,-1)})\;, \\
  \phi_{xy} &=
  \tfrac{1}{4\varepsilon^2}(\phi^{(1,1)}-\phi^{(-1,1)}
  -\phi^{(1,-1)}+\phi^{(-1,-1)})\;.
  \end{align*}
  %
  \item
  Advect--and--project using $P_2^+\/$. At time $t+\Delta t\/$ and at every
  grid point $(x\/,\,y) = (x_m\/,\,y_m)\/$, the values for $(\phi\/,\,
  D\phi\/,\, D^2\phi\/,\, \phi_{x\/x\/y}\/,\, \phi_{x\/y\/y}\/,\,
  \phi_{x\/x\/y\/y})\/$ must be computed from the solution at time $t\/$. To
  do so, use the approximate advection solver $\vec{\mathcal{X}}\/$ with
  analytic differentiation, to compute $(\phi\/,\, D\phi\/,\,D^2\phi)\/$ at
  the three points $(x_m\/,\,y_m)\/$ and $(x_m \pm \varepsilon\/,\,y_m)\/$.
  From this obtain
  $(\phi_{x\/x\/y}\/,\,\phi_{x\/y\/y}\/,\,\phi_{x\/x\/y\/y})\/$ at the grid
  point using $O(\varepsilon^2)\/$ accurate centered differences in $x\/$.
 \end{itemize}
  Error analysis. In both cases $P_1^+\/$ and $P_2^+\/$, the errors
  in the approximations are:
  $\mathcal{E} = O(\varepsilon^2) + O(\delta)\/$ for averages,
  $\mathcal{E} = O(\varepsilon^2)+O(\delta/\varepsilon)\/$ for
  first order centered $\varepsilon$-differences, and
  $\mathcal{E} = O(\varepsilon^2)+O(\delta/\varepsilon^2)\/$ for
  second order centered $\varepsilon$-differences.
  Thus the optimal choice for $\varepsilon\/$ is
  $\varepsilon = O(\delta^{1/4})\/$, which yields
  $\mathcal{E} = O(\delta^{1/2})\/$ for all approximations.
\end{example}
%
\begin{remark} \label{rem:epsilon:update:gridlines}
 In Remark~\ref{rem:analytic:update:gridlines} the scenario of a
 characteristic's foot point falling on a cell boundary is addressed.
 For $\varepsilon$-finite differences the analogous situation is more
 critical, since several characteristics, separated by $O(\varepsilon)\/$
 distances, are used. Whenever the foot points for these characteristics
 fall in different cells, the interpolant corresponding to one and the
 same cell must be used in the $\varepsilon$-finite differences. This is
 always possible, since the interpolants are $p$-$n$-polynomials, and thus
 defined outside the cells.
\end{remark}
%
Finally, we point out that,
when computing derivatives with $\varepsilon$-finite differences,
round-off error becomes important --- the more so the higher the
derivative. The error for an $s^{th}\/$ order derivative, with an order
$O(\varepsilon^r)\/$ accurate $\varepsilon$-finite difference, has the
form $\mathcal{E} = O(\varepsilon^r)+O(\delta/\varepsilon^s)\/$. The best
choice is then $\varepsilon = O(\delta^{1/(r+s)})\/$, which yields
$\mathcal{E} = O(\delta^{r/(s+r)})\/$. Hence, under some circumstances
it may be advisable to use high precision arithmetic and thus make
$\delta\/$ very small.
Given their relatively simple implementation, $\varepsilon$-finite
differences can be the best choice in many situations.

\subsection{Boundary Conditions} \label{subsec:Boundary:Conditions}
%
In this section we illustrate the fact that, due to their use of the
method of characteristics, jet schemes treat boundary conditions
naturally. For simplicity, we consider the two dimensional computational
domain $\Omega=[0,1]^2\/$, equipped with a regular square grid of
resolution $h = \frac{1}{N}\/$ --- i.e.\ the grid points are
$\vec{x}_{\vec{m}} = h\,\vec{m}\/$, where $\vec{m} \in \{0,\dots,N\}^2$.
In $\Omega\/$ we consider the advection equation
\begin{equation} \label{eqn:bc:01}
 \phi_t + u\,\phi_x + v\,\phi_y = 0\/,
 \quad \mbox{with initial conditions} \quad
 \phi(x\/,\,y\/,0) = \Phi(x\/,\,y)\/.
\end{equation}
Let us assume that $u > 0\/$ on both the left and right boundaries,
$v < 0\/$ on the bottom boundary, and $v > 0\/$ on the top boundary.
Hence the characteristics enter the square domain of computation
$\Omega\/$ only through the left boundary, where boundary conditions
are needed. For example: Dirichlet $\phi(0\/,\,y\/,t) = \xi(y\/,\,t)\/$
or Neumann $\phi_x(0\/,\,y\/,t) = \zeta(y\/,\,t)\/$.
We will also consider the special situation where $u \equiv 0\/$ on
the left boundary, in which case no boundary conditions are needed.

We now approximate \eqref{eqn:bc:01} with a jet scheme, with the time
step restricted by  $\Delta t < h/\Lambda\/$, where $\Lambda\/$ is an
upper bound for $\sqrt{u^2+v^2}\/$ --- so that for each grid point
$\norm{\xfoot-\vec{x}} < h\/$. Then $\xfoot \in \Omega\/$ for any node
$(x_m\/,\,y_n)\/$ with $m \neq 0\/$. Hence: at these nodes the data
defining the solution can be updated by the process described earlier
in this paper. We only need to worry about how to determine the data
on the nodes for which $x = 0\/$. There are three cases to describe.
\begin{itemize}
 \item[(i)]   Dirichlet boundary conditions. Then, on $x = 0\/$ we have:
 $\phi_t = \xi_t\/$ and $\phi_y = \xi_y\/$, while $\phi_x\/$ follows
 from the equation: $\phi_x = - (\phi_t + v\,\phi_y)/u\/$. Formulas for
 the higher order derivatives of $\phi\/$ at the left boundary can be
 obtained by differentiating the equation. For example, from
 \begin{equation*}
  \phi_{ty} + u_y\,\phi_x + u\,\phi_{xy} + v_y\,\phi_y + v\,\phi_{yy} = 0
 \end{equation*}
 we can obtain $\phi_{xy}\/$.
 \item[(ii)]  Neumann boundary conditions. Then, on $x = 0\/$, $\phi\/$
 satisfies $\phi_t + v\,\phi_y = -u\,\zeta\/$. This is a lower dimensional
 advection problem, which can be solved\footnote{Note that jet schemes can
    be easily generalized to problems with a source term. That
    is, replace Equation~\eqref{eq:linear_advection} by
    $\phi_t+\vec{v}\cdot\nabla\phi = S\/$, where $S$ is a known function.}
 to obtain $\xi = \phi(0\/,\,y\/,t)\/$. Higher order derivatives follow
 in a similar fashion.
 \item[(iii)] $u \equiv 0\/$ on the left boundary. Then
 $\xfoot \in \Omega\/$ for the nodes with $x=0\/$, so that the data
 defining the solution can be updated in the same way as for the nodes
 with $x > 0\/$.
\end{itemize}

\subsection{Advection and Function Spaces: Stability}
\label{subsec:adv:funspaces}
%
In this section we show how the interpretation of jet schemes as a process
of advect--and--project in function spaces, together with the fact that
Hermite interpolants minimize the stability functional in
Equation~\eqref{eq:functional}, lead to stability for jet schemes. We
provide a rigorous proof of stability in one space dimension for constant
coefficients. Furthermore, we outline (i) how the arguments could be
extended to higher space dimensions, and (ii) how the difficulties in the
variable coefficients case could be overcome. Note that, even though no
rigorous stability proofs in higher space dimensions are given here, the
numerical tests shown in \S~\ref{sec:numerical_results} indicate that
the presented jet schemes are stable in the 2-D variable coefficients case.
%
\subsubsection{Control over averages in 1-D Hermite interpolants}
\label{subsubsec:Control:Averages:1D:interp}
%
Let $k \in \mathbb{N}_0\/$, with $n=2\/k+1\/$. Let $\mathcal{H}\/$ be an
arbitrary $1$-$n$-polynomial, and define
\begin{equation} \label{eqn:stab:r_and_mu}
 r_k = r_k(x) = \frac{1}{(n+1)!}\,x^{k+1}\,(1-x)^{k+1}
\quad \mbox{and} \quad
 \mu_k = \mu_k(x) = r_k^{(k+1)}\/,
\end{equation}
where $f^{(j)}\/$ indicates the $j^{th}\/$ derivative of a function
$f=f(x)\/$. Integration by parts then yields
\begin{equation} \label{eqn:stab:mu_orthogonal_h}
 \int_0^1 \mu_k(x)\,\mathcal{H}^{(k+1)}(x)\/\ud{x} =
 (-1)^{k+1}
 \int_0^1 r_k(x)\,\mathcal{H}^{(n+1)}(x)\/\ud{x} = 0\/,
\end{equation}
where the boundary contributions vanish because of the $(k+1)^{st}$ order
zeros of $r_k\/$ at $x = 0\/$ and at $x = 1\/$. Integration by parts also
shows that, for any sufficiently smooth function $\phiSP = \phiSP(x)\/$
\begin{equation*}
 \int_0^1 \mu_k(x)\,\phiSP^{(k+1)}(x)\/\ud{x} =
 \prn{\mu_k^{(0)}\,\phiSP^{(k)}}\rule[-2ex]{0.2mm}{5ex}_{\,x=0}^{\,x=1} -
 \prn{\mu_k^{(1)}\,\phiSP^{(k-1)}}\rule[-2ex]{0.2mm}{5ex}_{\,x=0}^{\,x=1}
 + \dots + \int_0^1 \phiSP(x)\/\ud{x}\/,
\end{equation*}
where we have used that $\mu_k^{(k+1)} \equiv (-1)^{k+1}\/$. This can also
be written in the form
\begin{eqnarray}
 \int_0^1 \mu_k(x)\,\phiSP^{(k+1)}(x)\/\ud{x} -
 \int_0^1 \phiSP(x)\/\ud{x} & = &
 \prn{\mu_k^{(0)}\,\phiSP^{(k)}}\rule[-2ex]{0.2mm}{5ex}_{\,x=0}^{\,x=1} -
 \prn{\mu_k^{(1)}\,\phiSP^{(k-1)}}\rule[-2ex]{0.2mm}{5ex}_{\,x=0}^{\,x=1}
 \qquad \nonumber \\ && + \dots + (-1)^k
 \prn{\mu_k^{(k)}\,\phiSP^{(0)}}\rule[-2ex]{0.2mm}{5ex}_{\,x=0}^{\,x=1}\/.
 \label{eqn:stab:mu_and_phi}
\end{eqnarray}
In particular, assume now that $\mathcal{H}\/$ is the $n^{th}\/$ Hermite
interpolant for $\phiSP\/$ in the unit interval, and apply
\eqref{eqn:stab:mu_and_phi} to both $\mathcal{H}\/$ and $\phiSP\/$.
Then the resulting right hand sides in the equations are equal, so that
the left hand sides must also be equal. Thus, from
\eqref{eqn:stab:mu_orthogonal_h}, it follows that
\begin{equation} \label{eqn:stab:average}
 \int_0^1 \mathcal{H}(x)\/\ud{x} = \int_0^1 \phiSP(x)\/\ud{x} -
 \int_0^1 \mu_k(x)\,\phiSP^{(k+1)}(x)\/\ud{x}\/.
\end{equation}
This can be scaled, to yield a formula that applies to the $n^{th}\/$
Hermite interpolant for $\phiSP\/$ on any 1-D cell
$x_n \leq x \leq x_{n+1} = x_n + h\/$
\begin{equation} \label{eqn:stab:averageh}
 \int_{x_n}^{x_{n+1}} \mathcal{H}(x)\/\ud{x} =
 \int_{x_n}^{x_{n+1}} \phiSP(x)\/\ud{x} - h^{k+1}\,
 \int_{x_n}^{x_{n+1}} \mu_k\prn{\frac{x-x_n}{h}}\,
 \phiSP^{(k+1)}(x)\/\ud{x}\/.
\end{equation}
The following theorem is a direct consequence of this last formula
%
\begin{theorem} \label{thm:stab:averageh} 
 In one dimension, consider the computational domain
 $\Omega = \{x\,|\,a \leq x \leq b\}\/$ and the grid: $x_n = a + n\,h\/$
 --- where $0 \leq n \leq N\/$, $h = (b-a)/N\/$, and $N > 0\/$ is an
 integer. Then, for any function $\phiSP \in C^k\/$, with
 integrable $(k+1)^{st}\/$ derivative
 \begin{equation*}
 \int_a^b (P_k^{+}\phiSP)(x)\/\ud{x} = \int_a^b \phiSP(x)\/\ud{x} - \,
 h^{k+1}\,\int_a^b E_k\prn{\frac{x-a}{h}}\,\phiSP^{(k+1)}(x)\/\ud{x}\/,
 \end{equation*}
 where $E_k = E_k(z)\/$ is the periodic function of period one defined by
 $E_k(z) = \mu_k(z\,\mathrm{mod}\,1)\/$.
\end{theorem}
%
\subsubsection{Stability for 1-D constant advection}
\label{subsubsec:Stab:1D:constant:advection}
%
Using Theorems~\ref{thm:stab:averageh} and \ref{thm:minimize:seminorm}
we can now prove stability for jet schemes in one dimension, and with a
constant advection velocity. For simplicity we will also assume periodic
boundary conditions.
%
\begin{theorem} \label{thm:stab:1d_cc} 
 Under the same hypothesis as in Theorem~\ref{thm:stab:averageh},
 consider the constant coefficients advection equation
 $\phi_t + v\,\phi_x = 0\/$ in the computational domain $\Omega\/$, with
 periodic boundary conditions $\phi(b\/,\,t) = \phi(a\/,\,t)\/$, and
 initial condition $\phi(x\/,\,0) = \phiIC(0)\/$. Approximate the solution
 by the jet scheme defined by the advect--and--project process of
 \S~\ref{subsec:advect_project} and the projection $P_k^+\/$, with a time
 step $\Delta t\/$ satisfying $\Delta t \geq O(h^{k+1})\/$. Then this
 scheme is stable, in the sense described next. Let $\phi_n\/$ be the
 numerical solution at time $t_n = n\,\Delta t\/$, and consider a fixed
 time interval $0 \leq t_n \leq T\/$ --- for some fixed initial condition.
 Then:
 \begin{center}
   $\norm{\phi_n^{(\ell)}}_{\infty}\/$ and $\norm{\phi_n^{(k+1)}}_2\/$
   remain bounded as $h \to 0\/$, where $0 \leq \ell \leq k\/$.
 \end{center}
 Notation: here $\norm{\cdot}_p\/$ is the $L^p\/$ norm in $\Omega\/$, and
 $f^{(\ell)}\/$ denotes the $\ell^{th}\/$ derivative of a function $f\/$.
\end{theorem}
%
\begin{proof}
 For any Runge-Kutta scheme (and many other types of ODE solvers),
 the approximate advection operator is given by the shift operator
 \begin{equation}\label{thm:stab:1d_cc:eqn1}
  (A\,\phi)(x\/,\,t) = \phi(x - v\,\Delta t\/,\,t)\/,
 \end{equation}
 where we use the notation $A = A_{t+\Delta t\/,\,t}\/$ --- since
 $A_{t+\Delta t\/,\,t}\/$ is independent of time. Let the numerical solution
 at time $t = t_n = n\,\Delta t\/$ be denoted by $\phi_n\/$. Note that
 $\phi_n = (P_k^+ \circ A)^n \,\phiIC \in C^k\/$.

 From Theorem~\ref{thm:minimize:seminorm},
 Equation~\ref{thm:stab:1d_cc:eqn1}, and
 $\phi_{n+1} = (P_k^+ \circ A) \, \phi_n\/$ it follows that
 \begin{equation}\label{thm:stab:1d_cc:eqn2}
  \norm{\phi_{n+1}^{(k+1)}}_2 \leq \norm{\phi_n^{(k+1)}}_2
  \quad \Longrightarrow \quad
  \norm{\phi_{n}^{(k+1)}}_2 \leq \norm{\phiIC^{(k+1)}}_2\/.
 \end{equation}
 Furthermore, from Theorem~\ref{thm:stab:averageh}
 \begin{equation*}
  \mathcal{M}(\phi_{n+1}) = \mathcal{M}(\phi_n) - h^{k+1}\,
  \int_a^b E_k\prn{\frac{x-a}{h}}\,\phi_n^{(k+1)}(x-v\,\Delta t)\/\ud{x}\/,
 \end{equation*}
 where $\mathcal{M}\/$ denotes the integral from $a\/$ to $b\/$. From
 this, using the Cauchy-Schwarz inequality, we obtain the estimate
 \begin{equation}\label{thm:stab:1d_cc:eqn3}
  \abs{\mathcal{M}(\phi_{n+1})} \leq \abs{\mathcal{M}(\phi_n)} +
  h^{k+1}\,C\,\norm{\phi_n^{(k+1)}}_2\/,
 \end{equation}
 where
 \begin{equation*}
  C^2 = \int_a^b \prn{E_k\prn{\frac{x-a}{h}}}^2\/\ud{x}
      = (b-a) \int_0^1 \prn{E_k(x)}^2\/\ud{x} =
              \frac{b-a}{n+2}\,\prn{\frac{(k+1)!}{(n+1)!}}^2\/.
 \end{equation*}
 From \eqref{thm:stab:1d_cc:eqn3}, using
 \eqref{thm:stab:1d_cc:eqn2}, it follows that
 \begin{equation}\label{thm:stab:1d_cc:eqn4}
  \abs{\mathcal{M}(\phi_{n})} \leq \abs{\mathcal{M}(\phiIC)} +
  \frac{t_n}{\Delta t}\,h^{k+1}\,C\,\norm{\phiIC^{(k+1)}}_2\/.
 \end{equation}
 Now, for any of the (periodic) functions $g = \phi_n^{(\ell)}\/$ ---
 $0 \leq \ell \leq k\/$, we can write
 \begin{equation}\label{thm:stab:1d_cc:eqn5}
  g(x) = \overline{g} + \int_a^b G(x-y)\,g^{(1)}(y)\/\ud{y}\/,
 \end{equation}
 where $\overline{g}\/$ is the average value for $g\/$ and $G\/$ is the
 Green's function defined by: $G\/$ is periodic of period $b-a\/$, and
 $G(x) = \frac{1}{2} - \frac{x}{b-a}\/$ for $0 < x < b-a\/$. Note that
 $\overline{g} = 0\/$ for $1 \leq \ell \leq k\/$, and
 $\overline{g} = \frac{1}{b-a}\,\mathcal{M}(\phi_n)\/$ for $\ell = 0\/$.

 Now use \eqref{thm:stab:1d_cc:eqn2}, \eqref{thm:stab:1d_cc:eqn5},
 and the Cauchy-Schwarz inequality, to find an $h$-independent bound for
 $\norm{\phi_n^{(k)}}_{\infty}\/$ --- hence also for
 $\norm{\phi_n^{(k)}}_2\/$. Repeat the process for $\phi_n^{(k-1)}$, and
 so on --- all the way down to $\phi_n^{(0)}\/$. In the last step,
 \eqref{thm:stab:1d_cc:eqn4} is needed.
\end{proof}
%
Notice that Theorem~\ref{thm:stab:1d_cc} says nothing about the
$L^{\infty}\/$ norm of $\phi_n^{(k+1)}\/$. A natural question is then:
what can we say about $\phi_n^{(k+1)}\/$ --- beyond the statement in
Equation~\eqref{thm:stab:1d_cc:eqn2}? In particular, notice that there
are stricter bounds on the growth of derivatives than the one given by
Theorem~\ref{thm:minimize:seminorm} , since $\mathcal{F}$ is actually
minimized in each cell individually. Can these, as well results similar
to the one in Theorem~\ref{thm:stab:averageh}, be exploited to obtain
more detailed information about the behavior of the solutions (and
derivatives) given by jet schemes? This is something that we plan to
explore in future work.
%
\subsubsection{Stability for 1-D variable advection}
\label{subsubsec:stab:1D:var:advection}
%
An extension of this proof to the case of variable advection will be
considered in future work.
%
\subsubsection{Stability for higher dimensions and constant coefficients}
\label{subsubsec:stability:multiD:const:coeff}
%
In several space dimensions, assume that the advection velocity does not
involve rotation --- specifically: the advected grid hyperplanes remain
parallel to the coordinate hyperplanes. Then all the derivatives needed to
compute the stability functional $\mathcal{F}\/$ (as well as all the
derivatives needed in the proof of Theorem~\ref{thm:minimize:seminorm})
remain defined after advection. Thus Theorem~\ref{thm:minimize:seminorm}
can be used, as in the proof of Theorem~\ref{thm:stab:1d_cc}, to control
the growth of $\norm{\partial^{\vec{\beta}}\,\phi}_2\/$ --- where
$\vec{\beta}\/$ is as in Equation~\eqref{eq:functional}. In particular,
when the advection velocity is constant,
\begin{equation}\label{eqn:stab:manyD}
 \norm{\partial^{\vec{\beta}}\,\phi}_2 \leq
 \norm{\partial^{\vec{\beta}}\,\phiIC}_2
\end{equation}
follows. Just as in the one dimensional case, \eqref{eqn:stab:manyD}
alone is not enough to guarantee stability. For example: in the case with
periodic boundary conditions, control over the growth of
$\partial^{\vec{\beta}}\,\phi\/$ still leaves the possibility of unchecked
growth of partial means, i.e. components of the solution of the form
$\phi = \sum f_j(\vec{x})\/$, where $f_j\/$ does not depend on the
variable $x_j\/$. However, the Hermite interpolant for any such component
reduces to the Hermite interpolant of a lower dimension, to which a lower
dimensional version of the stability functional $\mathcal{F}\/$ applies
--- thus bounding their growth. The appropriate generalization of
Theorem~\ref{thm:stab:averageh} to control the partial means of $\phi$
is the subject of current research.
%
\subsubsection{Stability for higher dimensions and variable coefficients}
\label{subsubsec:stability:multiD:var:coeff}
%
In the presence of rotation, Theorem~\ref{thm:minimize:seminorm} fails.
Then, after advection, the derivatives involved are only defined in the
sense of Definition~\ref{def:define:pointwise}, and the integrations by
parts used to prove Theorem~\ref{thm:minimize:seminorm} are no longer
justified. On the other hand (see the discussion at the beginning of
\S~\ref{subsec:projections_function_spaces}), the jumps in the derivatives
that cause the failure are small. Hence we conjecture that, in this case,
a ``corrected'' version of Theorem~\ref{thm:minimize:seminorm} will state
that $\mathcal{F}\left[\mathcal{H}_{\phi}\right] \leq \mathcal{F}[\phi]
+ C(h)\/$, where $C(h)\/$ is a small correction that depends on the
smallness of the jumps.
In this case a simultaneous proof of stability and convergence might be
possible --- simultaneous because the ``small jumps'' property is valid
only as long as the numerical solution is close to the actual solution.

\section{Numerical Results}
\label{sec:numerical_results}
%
As a test for both the accuracy and the performance of jet schemes, we
consider a version of the classical ``vortex in a box'' flow
\cite{BellColellaGlaz1989,LeVeque1996}, adapted as follows. On the
computational domain $(x\/,\,y) \in [0,1]^2\/$, and for
$t \in [0\/,\,t_\text{final}]\/$, we consider the linear advection
equation~\eqref{eq:linear_advection} with the velocity field
\begin{equation} \label{eq:vortex_in_box_velocity_field}
  \vec{v}(x\/,\,y\/,\,t) = \cos\prn{\tfrac{\pi t}{T}}
  \begin{pmatrix}
     \phantom{-}\sin^2(\pi x)\,\sin(2\pi y) \\
               -\sin(2\pi x) \,\sin^2(\pi y)
  \end{pmatrix}.
\end{equation}
%
\begin{figure}
\begin{minipage}[b]{.80\textwidth}
\centering
\includegraphics[width=.62\textwidth]{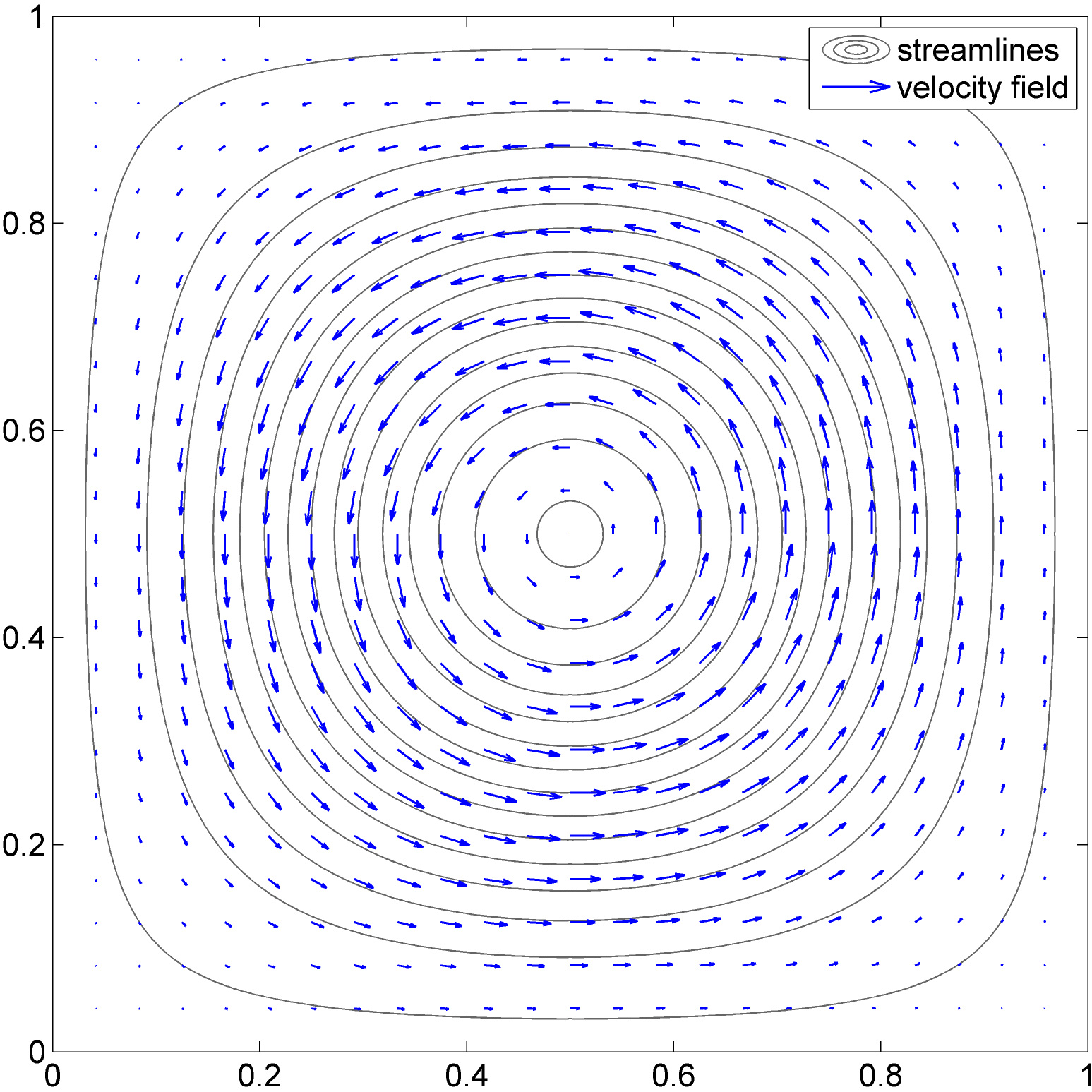}
\caption{Streamlines and quiver plot of the velocity field used for the test cases.}
\label{fig:swirl_quiver}
\end{minipage}
\end{figure}
This velocity field at $t=0$ is shown in Figure~\ref{fig:swirl_quiver}.
The maximum speed that ever occurs is $1\/$. We prescribe periodic
boundary conditions on all sides of the computational domain, and
provide periodic initial conditions --- note that the velocity field
\eqref{eq:vortex_in_box_velocity_field} is $C^\infty\/$ everywhere when
repeated periodically beyond $[0\/,\,1]^2\/$. This test is a mathematical
analog of the famous ``unmixing'' experiment, presented by Heller
\cite{Heller1960}, and popularized by Taylor \cite{FluidMixingMovie}. It
models the passive advection of a solute concentration by an
incompressible fluid motion, on a time scale where diffusion can be
neglected. The velocity field swirls the concentration forth and back
(possibly multiple times) around the center point
$(\frac{1}{2}\/,\,\frac{1}{2})\/$ in such a fashion that the
equi-concentration contours of the solution become highly elongated at
maximum stretching. The parameters $T\/$ and $t_\text{final}\/$ determine
the amount of maximum deformation and the number of swirls. In all tests,
we choose $t_\text{final} = \ell\,T\/$, where $\ell \in \mathbb{N}$. This
implies that $\phi(x,y,t_\text{final}) = \phi(x\/,\,y\/,\,0)\/$,
i.e.\ at the end of each computation, the solution returns to its initial
state. Below, we compare the accuracy and convergence rates of jet
schemes with WENO methods (\S~\ref{subsec:numerics_accuracy}),
investigate the accuracy and stability of different versions of jet
schemes (\S~\ref{subsec:numerics_stability}), demonstrate the
performance of jet schemes on a benchmark test
(\S~\ref{subsec:numerics_performancs}), and finally compare the
computational cost and efficiency of all the numerical schemes
considered (\S~\ref{subsec:numerics_cost}).

\subsection{Test of Accuracy and Convergence Rate}
\label{subsec:numerics_accuracy}
We first test the relative accuracy and convergence rate of jet schemes
in comparison with classical WENO approaches. For this test, we choose
$t_\text{final} = T = 1\/$ and initial conditions
$\phi(x\/,\,y\/,\,0) = \cos(2\pi x)\cos(4\pi y)\/$. This choice of
parameters yields a moderate deformation, and the smallest structures in
the solution are well resolved for $h\le\sfrac{1}{50}$. In all error
convergence tests, the error with respect to the true solution at
$t_\text{final}\/$ is measured in the $L^\infty\/$ norm.

The results of the numerical error analysis are shown in
Figure~\ref{fig:convergence_1}. The jet schemes considered here are the
ones based on the projection $P_k^+\/$, given in
Definition~\ref{def:projection_partial}. Specifically, we consider a
bi-linear scheme ($k=0\/$), denoted by black dots; a bi-cubic scheme
($k=1\/$), implemented using analytical differentiation for the partial
$1$-jet, as described in \S~\ref{subsubsec:Analytical_Diff}, marked by
black squares; and the bi-quintic scheme ($k=2\/$) described in the
second bullet point in Example~\ref{ex:epsilon:FD}, denoted by black
triangles.

As reference schemes, we consider the classical WENO finite difference schemes described in \cite{JiangShu1996}. We use schemes of orders $1\/$, $3\/$, and $5\/$, constructed as follows. Unless otherwise noted, time stepping is done with the maximum time step that stability admits. The first order version is the simple 2-D upwind scheme, using forward Euler in time (with a time step $\Delta t = \frac{1}{\sqrt{2}}h\/$). The time stepping of the third order WENO is done using the Shu-Osher scheme \cite{ShuOsher1988} (with $\Delta t = h\/$). Lacking a simple fifth order SSP Runge-Kutta scheme \cite{Gottlieb2005,GottliebKetchesonShu2009}, the fifth order WENO is advanced forward in time using the same third order Shu-Osher scheme, however, with a time step $\Delta t = h^\frac{5}{3}\/$. This yields a globally order $O(h^5)\/$ scheme. For both WENO3 and WENO5, the parameters $\varepsilon = 10^{-6}\/$ and $p = 2$\/ (in the notation of \cite{JiangShu1996}) are used. These choices are commonly employed when WENO is used in a black-box fashion (i.e.\ without adapting the parameters to the specific solution). In Figure~\ref{fig:convergence_1}, the numerical errors incurred with the WENO schemes are shown by gray curves and symbols, namely: dots for upwind; squares for WENO3; and triangles for WENO5.

This test demonstrates the potential of jet schemes in a remarkable
fashion. While the first order jet scheme is very similar to the 2-D
upwind method, the jet schemes of orders $3\/$ and $5\/$ are strikingly
more accurate than their WENO counterparts. For equal resolution $h\/$,
the errors for the bi-cubic jet scheme are about 100 times smaller than
with WENO3, and for the bi-quintic jet scheme, the errors are about
1000 times smaller than with WENO5. This implies that a high order jet
scheme achieves the same accuracy as a WENO scheme (of the same order)
with a resolution that is about $4\/$ times as coarse. This reflects
the fact that jet schemes possess subgrid resolution, i.e.\ structures
of size less than $h\/$ can be represented due to the high degree
polynomial interpolation \cite{NaveRosalesSeibold2010}.

\subsection{Accuracy and Stability of Grid-Based Finite Difference
Schemes} \label{subsec:numerics_stability}
In a second test, the accuracy and stability of schemes that use grid-based
finite difference reconstructions of the higher derivatives are tested. As
in the test described in \S~\ref{subsec:numerics_accuracy}, we choose the
deformation $T=1\/$, however, now $t_\text{final} = 20\/$, i.e.\ the
solution is moved back and forth multiple times. Through this approach,
enough time steps are taken so that unstable approaches in fact show
their instabilities.

We compare jet schemes of order 3 ($k=1\/$) and order 5 ($k=2\/$), each in
two versions: first, we consider the approaches based on the projection
$P_k^+\/$ (see Definition~\ref{def:projection_partial}), to which the
stability arguments in \S~\ref{subsec:adv:funspaces} apply. Second, we
implement schemes that construct the partial $k$-jet from the total
$k$-jet, using optimally local grid-based finite difference
approximations. As described in \S~\ref{subsec:optimally:local}, these
types of schemes are generally less costly than approaches based on
$P_k^+\/$, and therefore worth investigating (see also
\S~\ref{subsec:numerics_cost}).

Specifically, we consider a bi-cubic scheme that tracks the total
$1$-jet, and uses the grid-based finite difference approximation of
$\phi_{xy}\/$ described in Example~\ref{ex:optimal:local:approx}.
Furthermore, we implement a bi-quintic scheme that tracks the total
$2$-jet, and approximates $\phi_{xxy}\/$, $\phi_{xyy}\/$, and
$\phi_{xxyy}\/$ using grid-based finite differences, as follows. In a
cell $Q = [0\/,\,h]^2$, with the index $\vec{q}\in\{0\/,\,1\}^2$
corresponding to the vertex $\vec{x}_{\vec{q}} = \vec{q}\,h\/$, the
derivatives at the vertex $(0,0)$ are approximated by:
\begin{align*}
  \phi_{xxy}^{(0,0)} = & \;
    \tfrac{1}{h^2}\/\left( - \phi_x^{(0,0)} +
    \phi_x^{(1,0)} + \phi_x^{(0,1)} - \phi_x^{(1,1)} \right) +
    \tfrac{6}{h^2}\/\left( - \phi_y^{(0,0)} + \phi_y^{(1,0)} \right)
    \\ + & \;
    \tfrac{1}{2h}\/\left( - \phi_{xx}^{(0,0)} - \phi_{xx}^{(1,0)} +
    \phi_{xx}^{(0,1)} + \phi_{xx}^{(1,1)} \right) +
    \tfrac{1}{h}\/\left( - 4\, \phi_{xy}^{(0,0)} - 2\,
    \phi_{xy}^{(1,0)} \right)\/,
  \\
  \phi_{xyy}^{(0,0)} = & \;
    \tfrac{6}{h^2}\/\left( - \phi_x^{(0,0)} + \phi_x^{(0,1)} \right) +
    \tfrac{1}{h^2}\/\left( - \phi_y^{(0,0)} + \phi_y^{(1,0)} +
    \phi_y^{(0,1)} - \phi_y^{(1,1)} \right)
    \\ + & \;
    \tfrac{1}{h}\/\left( - 4\,\phi_{xy}^{(0,0)} - 2\,\phi_{xy}^{(0,1)}
    \right) + \tfrac{1}{2h}\/\left( -\phi_{yy}^{(0,0)} +
    \phi_{yy}^{(1,0)} - \phi_{yy}^{(0,1)} + \phi_{yy}^{(1,1)} \right)\/,
  \\
  \phi_{xxyy}^{(0,0)} = & \;
    \tfrac{6}{h^3}\/\left( \phi_x^{(0,0)} - \phi_x^{(1,0)} -
    \phi_x^{(0,1)} + \phi_x^{(1,1)}\right) +
    \tfrac{6}{h^3}\/\left( \phi_y^{(0,0)} - \phi_y^{(1,0)} -
    \phi_y^{(0,1)} + \phi_y^{(1,1)} \right)
    \\ + &
    \tfrac{1}{h^2}\/\left( 7\,\phi_{xy}^{(0,0)} - \phi_{xy}^{(1,0)} -
    \phi_{xy}^{(0,1)} - 5\,\phi_{xy}^{(1,1)} \right)\/.
\end{align*}
It can be verified by Taylor expansion that these are $O(h^{6-s})\/$
accurate approximations to the respective derivatives of order $s\/$.
The corresponding approximations at the vertices $(1,0)\/$,
$(0,1)\/$ and $(1,1)\/$ are obtained by symmetry.

The convergence errors for these jet schemes are shown in
Figure~\ref{fig:convergence_2}. As already observed in
\S~\ref{subsec:numerics_accuracy}, the bi-cubic (shown by black squares)
and bi-quintic (denoted by black triangles) jet schemes that are based on
tracking the partial $k$-jet are stable and $(2k+1)^{st}\/$ order accurate.
The bi-cubic jet scheme that approximates $\phi_{xy}\/$ by grid-based
finite differences (indicated by open squares) is stable and third order
accurate as well, albeit with a larger error constant than the ``pure''
version based on $P_1^+\/$. In contrast, the bi-quintic jet scheme that
approximates $\phi_{xxy}\/$, $\phi_{xyy}\/$, and $\phi_{xxyy}\/$ by
grid-based finite differences (denoted by open triangles) is unstable.
Interestingly, the instability is extremely mild: even though many time
steps are taken in this test, the instability only shows up for
$h<0.008\/$. This phenomenon is both alarming and promising. Alarming,
because it demonstrates that with jet schemes that are not based on the
partial $k$-jet, stability is not assured, and generalizations of the
stability results given in \S~\ref{subsec:adv:funspaces} are needed.
Promising, because it is plausible that these jet schemes could be
stabilized without significantly having to diminish their accuracy.

\begin{figure}[p]
  \begin{minipage}[b]{.56\textwidth}
    \hspace{-.4em}\includegraphics[width=1.02\textwidth]{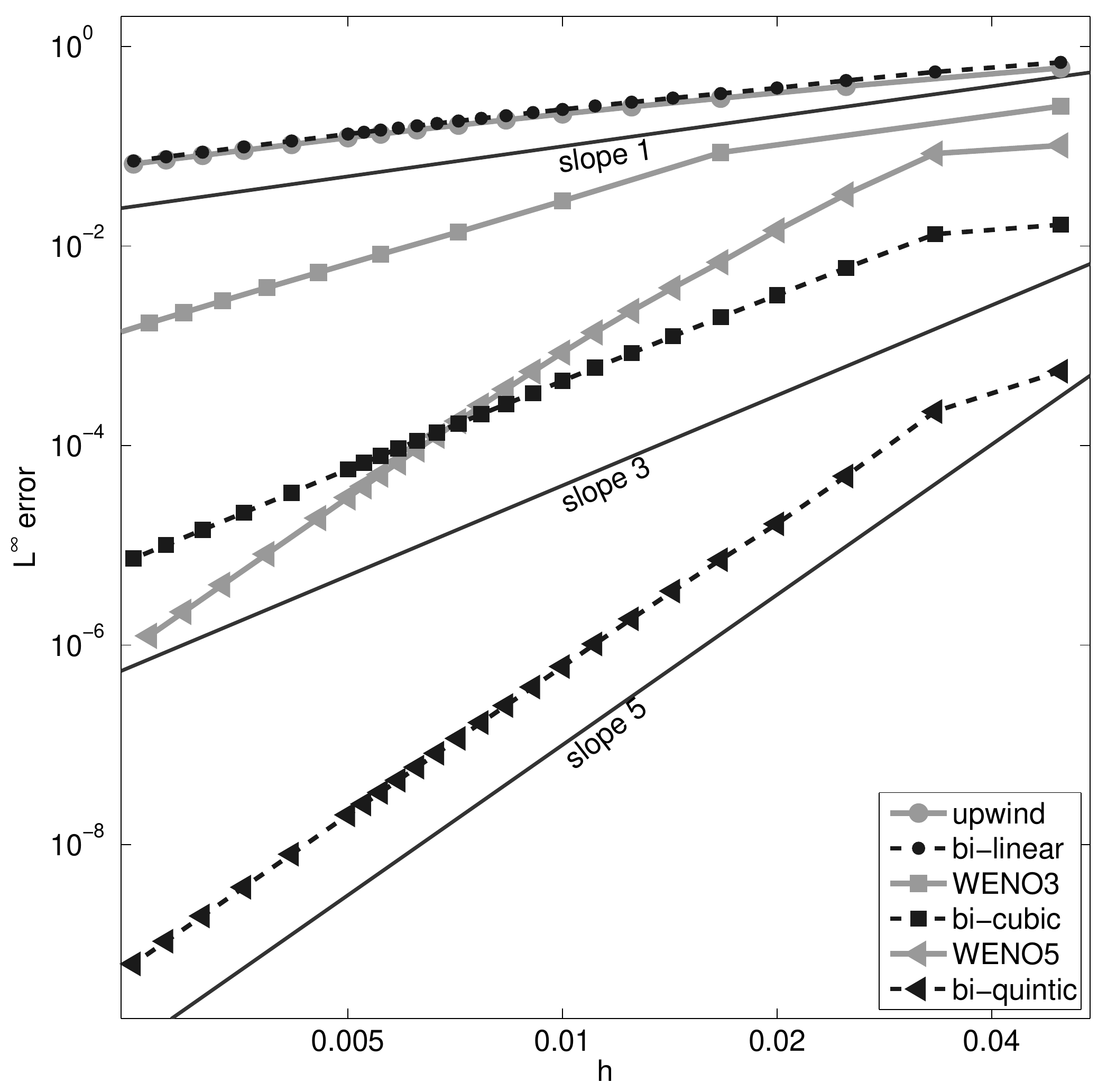}
    \vspace{-1.6em}
    \caption{Numerical convergence rates for jet schemes of orders
             $1\/$, $3\/$, and $5\/$, in comparison with WENO schemes
             of the same orders.}
    \label{fig:convergence_1}

    \vspace{2em}
    \hspace{-.4em}\includegraphics[width=1.02\textwidth]{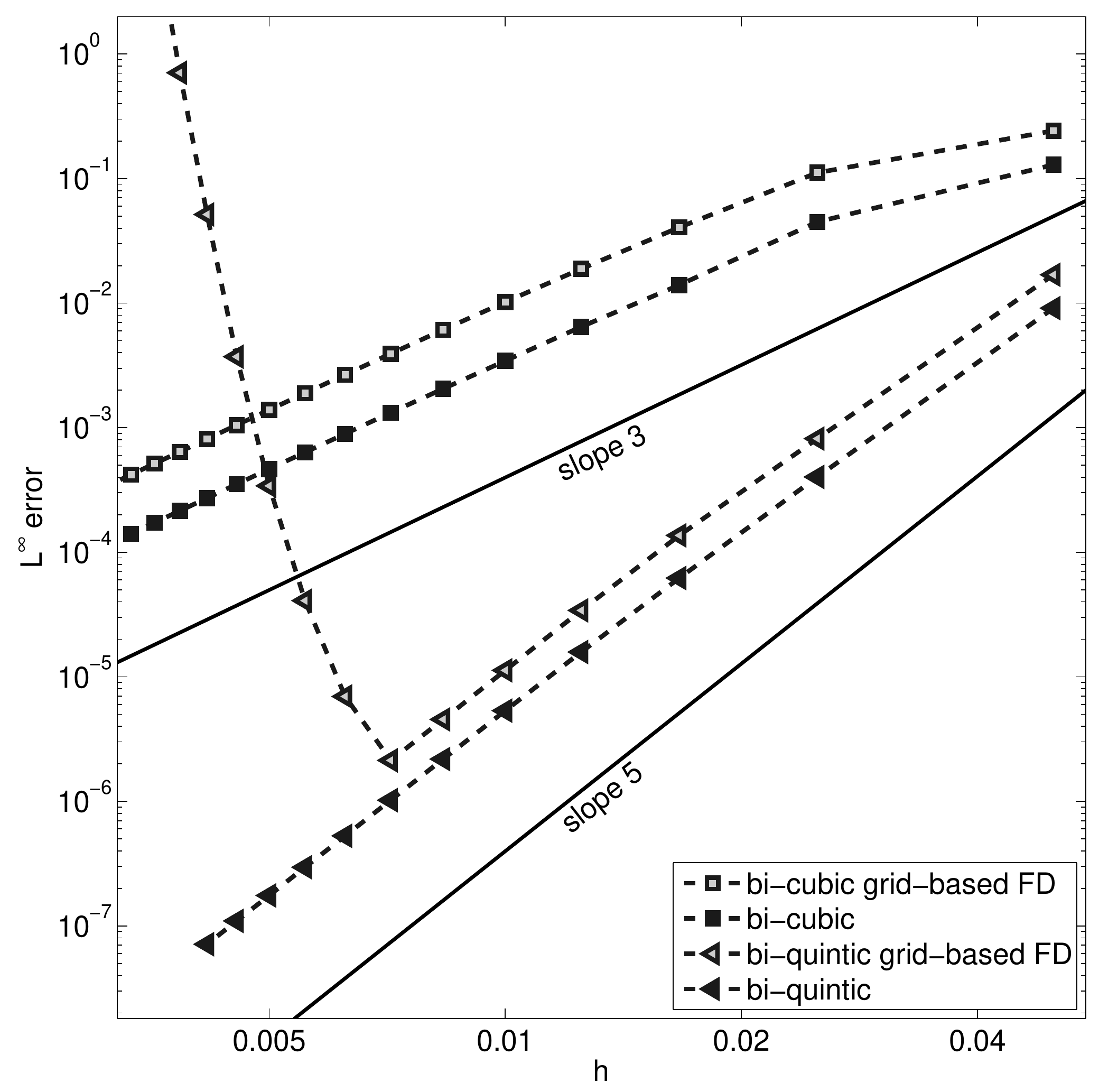}
    \vspace{-1.6em}
    \caption{Numerical convergence rates for jet schemes of orders $3\/$
           and $5\/$, comparing schemes based on the full partial jet
           with schemes using grid-based finite differences. The latter
           type of approach turns out unstable for the order $5\/$
           jet scheme.}
    \label{fig:convergence_2}
    \vspace{1.2em}
  \end{minipage}
 \hfill
  \begin{minipage}[b]{.40\textwidth}
    \includegraphics[width=\textwidth]{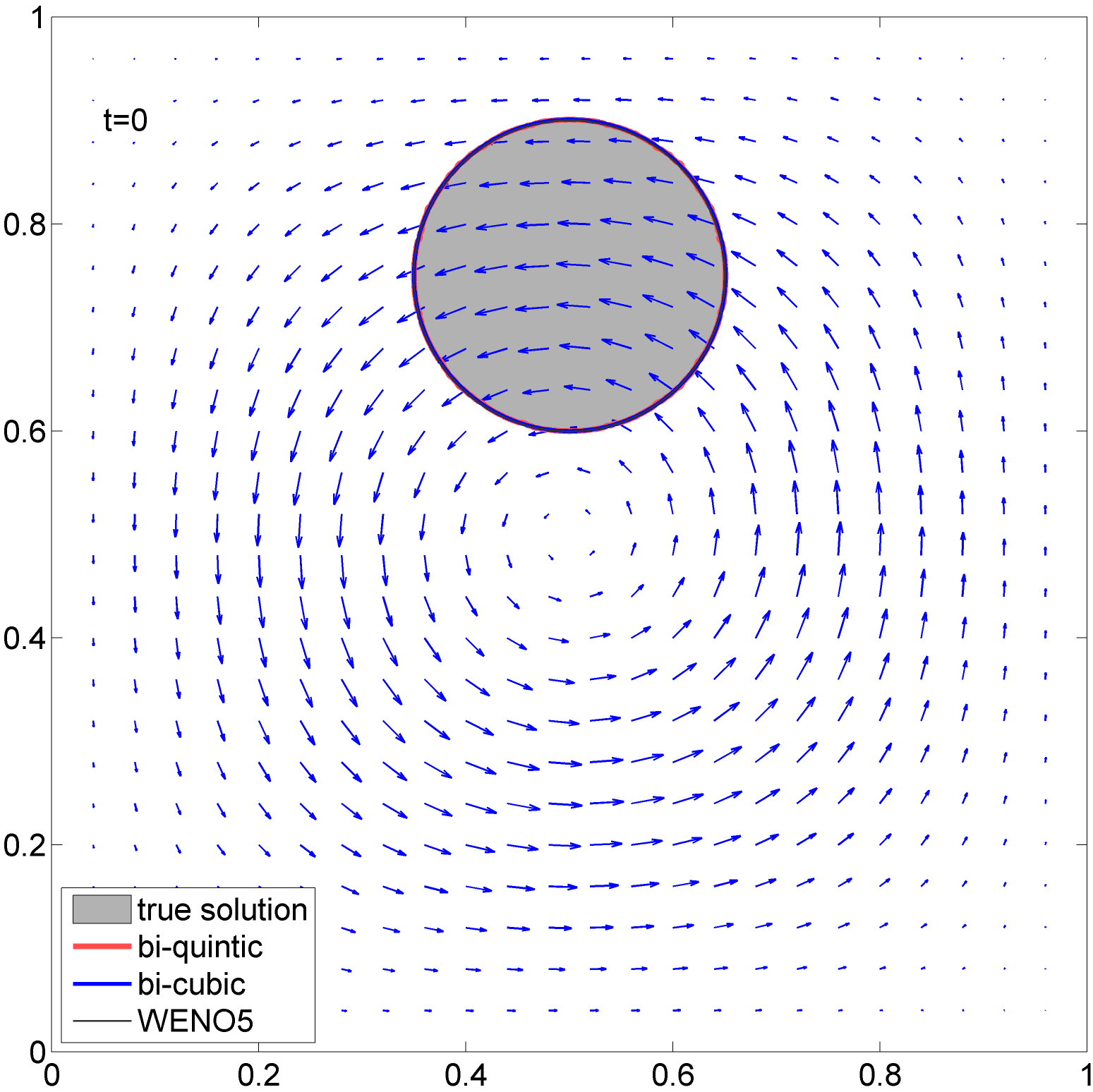}
    \vspace{-1.2em}
    \caption{Swirl test: velocity field and initial conditions.}
    \label{fig:swirl_t00}

    \vspace{1.1em}
    \includegraphics[width=\textwidth]{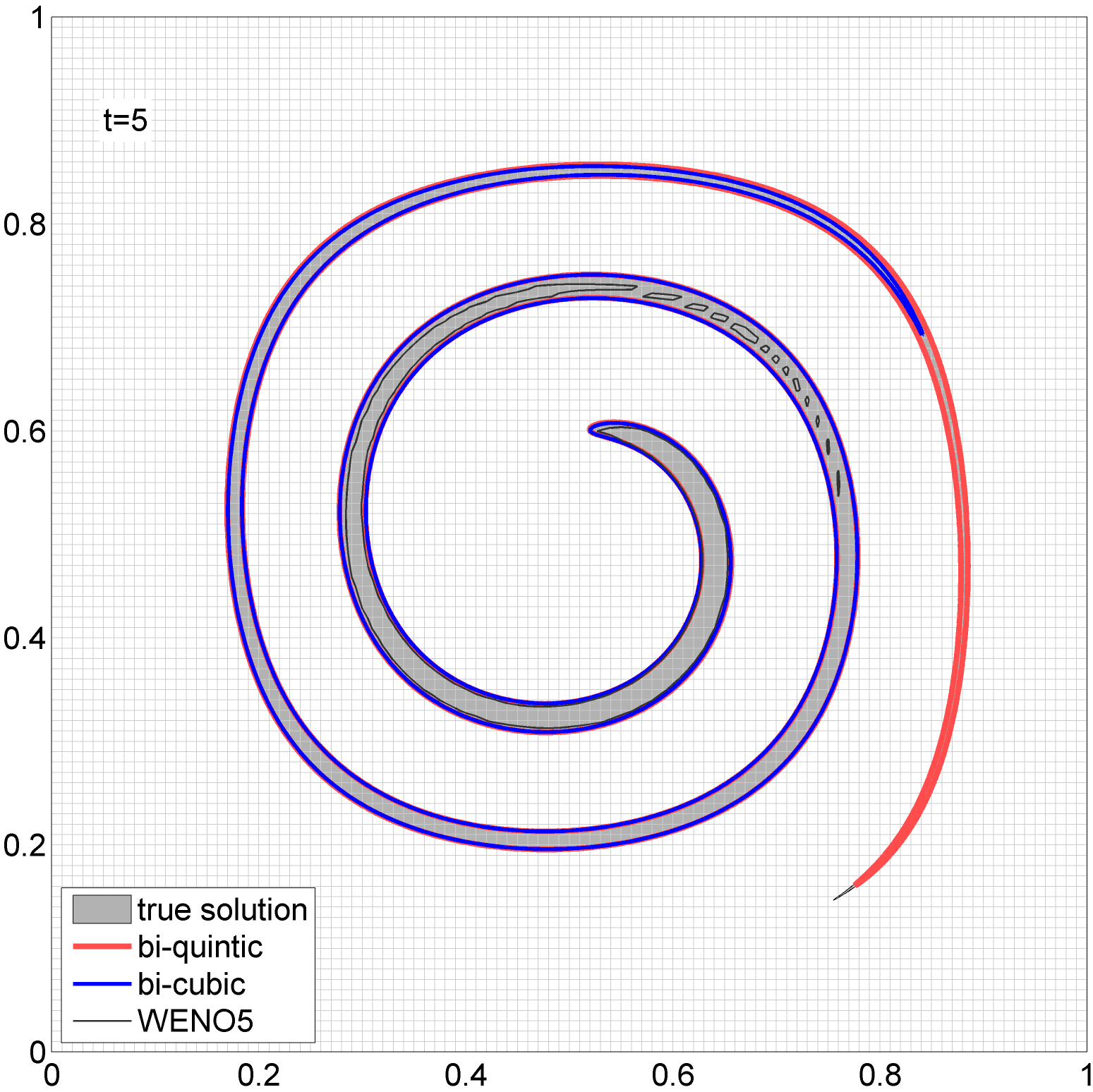}
    \vspace{-1.2em}
    \caption{Swirl test: $\phi=\phi_c\/$ contour at maximum deformation
             ($t=\frac{T}{2}\/$).}
    \label{fig:swirl_t05}

    \vspace{1.1em}
    \includegraphics[width=\textwidth]{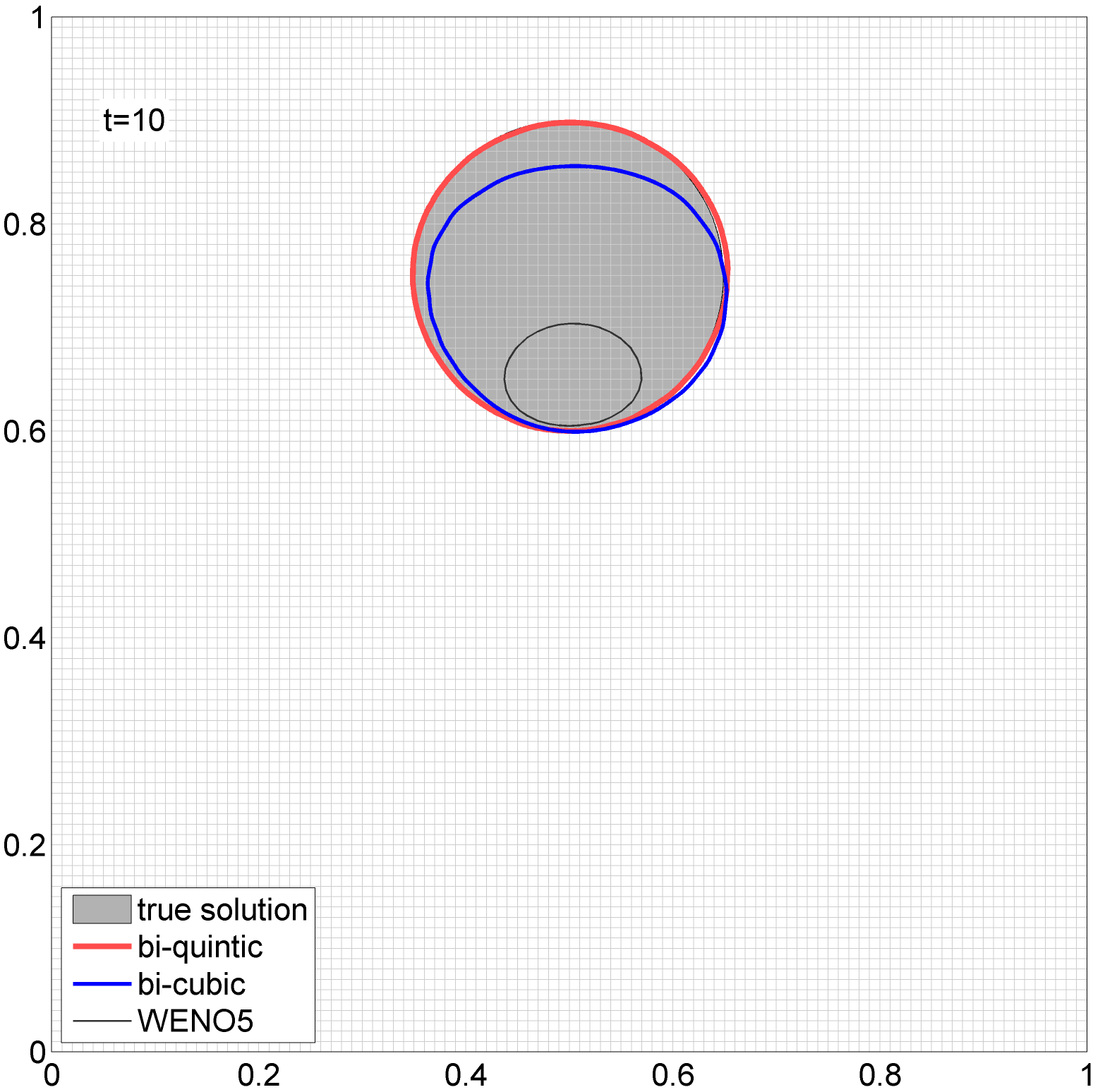}
    \vspace{-1.2em}
    \caption{Swirl test: $\phi=\phi_c\/$ contour at the final time
             ($t=T\/$).}
    \label{fig:swirl_t10}
  \end{minipage}
\end{figure}

\subsection{Test of Performance on a Level-Set-Type Example}
\label{subsec:numerics_performancs}
In this test we assess the practical accuracy of the considered numerical
approaches, by following a curve of equi-concentration of the solution.
The parameters are now $t_\text{final} = T = 10\/$, and the initial
conditions are given by a periodic Gaussian hump:
$\phi(x\/,\,y\/,\,0) = \sum_{i\/,\,j\in\mathbb{Z}} g(x-i\/,\,y-j)\/$,
where $g(x\/,\,y) = \exp(-10((x-x_0)^2+(y-y_0)^2))\/$ with
$x_0 = 0.5\/$ and $y_0 = 0.75\/$. We examine the time-evolution of the
contour
$\Gamma(t) = \{(x\/,\,y)\in\Omega:\phi(x\/,\,y\/,\,t)=\phi_c\}\/$ that
corresponds to the concentration $\phi_c = \exp(-10\,r^2)\/$, where
$r = 0.15\/$. The initial (and final) contour $\Gamma(0)\/$ is almost an
exact circle of radius $r\/$, centered at $(x_0\/,\,y_0)\/$. At maximum
deformation, $\Gamma(\tfrac{T}{2})\/$ is highly elongated.

This test is well-known in the area of level set approaches
\cite{OsherSethian1988}. For these, only one specific contour is of
interest, and it is common practice to modify the other contours, e.g.\
by adding a reinitialization equation \cite{SussmanSmerekaOsher1994}, or
by modifying the velocity field away from the contour of interest, using
extension velocities \cite{AdalsteinssonSethian1999}. Since here the
interest lies on more general advection problems
\eqref{eq:linear_advection}, such as the convection of concentration
fields, no level-set-method-specific modifications are considered.
However, this does not mean that these procedures cannot be applied in
the context of jet schemes. In fact, proper combinations of jet schemes
with reinitialization are the subject of current research.

We apply the bi-cubic and the bi-quintic jet schemes, as well as the
WENO5 scheme (all as described in \S~\ref{subsec:numerics_accuracy}), on
a grid of resolution $h=\sfrac{1}{100}\/$. The resulting equi-concentration
contours are shown in Figure~\ref{fig:swirl_t00} (initial conditions),
Figure~\ref{fig:swirl_t05} (maximum deformation), and
Figure~\ref{fig:swirl_t10} (final state). The contour obtained with WENO5
is thin and black, the contour obtained with the bi-cubic jet scheme is
medium thick and blue, and the contour obtained with the bi-quintic jet
scheme is thick and red. The true solution (approximated with high
accuracy using Lagrangian markers), is shown as a gray patch. For the
considered resolution, both jet schemes yield more accurate results than
WENO5. In fact, the fifth order jet scheme yields an almost flawless
approximation to the true solution, on the scale of interest. As alluded
to in \S~\ref{subsec:numerics_stability}, the subgrid resolution of the
jet schemes is of great benefit in resolving the thin elongated structure.

\subsection{Computational Cost and Efficiency} \label{subsec:numerics_cost}
In order to get an impression of the relative computational cost and efficiency of the presented schemes, we measure the CPU times that various versions of jet schemes and WENO require to perform given tasks. We use the same test as in \S~\ref{subsec:numerics_accuracy}, and apply the following versions of jet schemes: (a) jet schemes of orders 1, 3, and 5, that are based on the partial $k$-jet; and (b) jet schemes of orders 3 and 5, that are based on the total $k$-jet and use grid-based finite difference approximations for the missing derivatives of the partial $k$-jet (see \S~\ref{subsec:numerics_stability} for a description of these schemes). In addition, we consider the following versions of WENO and finite difference schemes to serve as reference methods: simple linear 2-D upwinding with forward Euler in time (with $\Delta t = \frac{1}{\sqrt{2}}h\/$), WENO3 \cite{JiangShu1996} advanced with the third order Shu-Osher method \cite{ShuOsher1988}, and WENO5 with the Cash-Karp method \cite{CashKarp1990} used for the time stepping. Note that unlike the case where convergence is tested (see \S~\ref{subsec:numerics_accuracy}), for a fair comparison of CPU times a true fifth order time stepping scheme must be used. While the Cash-Karp method is not SSP, for the considered test case no visible oscillations are observed. Furthermore, in order to give a fair treatment to WENO, one must be careful with the choice of the parameter $\varepsilon$ that WENO schemes require \cite{JiangShu1996}. The meaning of the parameter is that for $\varepsilon\ll h^2$, spurious oscillations are minimized at the expense of a degradation of convergence rate, while for $\varepsilon\gg h^2$, the schemes reduce to (unlimited) linear finite differences of the respective order. We therefore consider two extremal cases for both WENO3 and WENO5: one with a tiny $\varepsilon = 10^{-10}$, and one that is linear finite differences (called ``linear FD3'' and ``linear FD5'') with upwinding (with all limiting routines removed). Practical WENO implementations will be somewhere in between these two extremal cases.

\begin{table}
\begin{tabular}{|c|l|c|c|}
\hline
Order & Approach                         & $L^{\infty}$ error   & CPU time/sec. \\
\hline \hline
1 & bi-linear jet scheme                 & $1.69\times 10^{-1}$ & \phantom{00}0.400 \\
1 & upwind                               & $1.92\times 10^{-1}$ & \phantom{00}0.456 \\
\hline
3 & bi-cubic jet scheme                  & $1.35\times 10^{-4}$ & \phantom{00}7.91\phantom{0} \\
3 & bi-cubic jet scheme grid-based FD    & $2.31\times 10^{-4}$ & \phantom{00}5.99\phantom{0} \\
3 & WENO3/RK3 ($\varepsilon=10^{-10}$)   & $1.21\times 10^{-2}$ & \phantom{00}1.61\phantom{0} \\
3 & linear FD3/RK3                       & $1.54\times 10^{-3}$ & \phantom{00}1.15\phantom{0} \\
\hline
5 & bi-quintic jet scheme                & $8.23\times 10^{-8}$ & 103\phantom{.000} \\
5 & bi-quintic jet scheme grid-based FD  & $1.32\times 10^{-7}$ & \phantom{0}36.3\phantom{00} \\
5 & WENO5/RK5 ($\varepsilon=10^{-10}$)   & $1.25\times 10^{-4}$ & \phantom{0}19.3\phantom{00} \\
5 & linear FD5/RK5                       & $2.15\times 10^{-5}$ & \phantom{0}19.1\phantom{00} \\
\hline
\end{tabular}
\vspace{.7em}
\caption{CPU times for a fixed resolution $h = \sfrac{1}{150}$ with various jet schemes and WENO schemes.}
\label{table:cpu_times}
\vspace{-1em}
\end{table}

As a first comparison of the computational costs of these numerical approaches, we run all of them for a fixed resolution ($h = \sfrac{1}{150}\/$), and record the resulting $L^{\infty}$ errors and CPU times, as shown in Table~\ref{table:cpu_times}. One can see a clear separation between methods of different orders both in terms of accuracy and in terms of CPU times. When comparing approaches of the same order one can observe that for identical resolutions, jet schemes incur a significantly larger computational effort than WENO and linear finite difference schemes. However, this increased cost comes at the benefit of a significant increase in accuracy. Similarly, it is apparent that the use of grid-based finite difference approximations for parts of the $k$-jet can reduce the cost of jet schemes significantly. However, this tends to come at an expense in accuracy, in addition to the fact that the resulting scheme could become unstable, as demonstrated in \S~\ref{subsec:numerics_stability}. In fact, the CPU times for the bi-quintic jet scheme with grid-based finite differences shown in Table~\ref{table:cpu_times} must be interpreted as a guideline only, since without further stabilization this version of the method is of no practical use.

\begin{table}
\begin{tabular}{|c|l|c|c|c|}
\hline
Order & Approach & Res.~$h$ & $L^{\infty}$ error & time/sec. \\
\hline \hline
1 & bi-linear jet scheme
  & $\sfrac{1}{1600}$ & $1.89\times 10^{-2}$ ${}^{(*)\!\!}$ & 441 ${}^{(*)\!\!\!\!}$ \\
1 & upwind
  & $\sfrac{1}{1600}$ & $2.28\times 10^{-2}$ ${}^{(*)\!\!}$ & 482 ${}^{(*)\!\!\!\!}$ \\
\hline
3 & bi-cubic jet scheme
  & $\sfrac{1}{167}$  & $9.86\times 10^{-5}$ & \phantom{00}11.4\phantom{00} \\
3 & bi-cubic jet scheme grid-based FD
  & $\sfrac{1}{200}$  & $9.86\times 10^{-5}$ & \phantom{00}13.0\phantom{00} \\
3 & WENO3/RK3 ($\varepsilon=10^{-10}$)
  & $\sfrac{1}{1350}$ & $9.61\times 10^{-5}$ & 1110\phantom{.000} \\
3 & linear FD3/RK3
  & $\sfrac{1}{380}$  & $9.65\times 10^{-5}$ & \phantom{00}18.7\phantom{00} \\
\hline
5 & bi-quintic jet scheme
  & $\sfrac{1}{35}$   & $9.76\times 10^{-5}$ & \phantom{000}1.55\phantom{0} \\
5 & bi-quintic jet scheme grid-based FD
  & $\sfrac{1}{40}$   & $8.62\times 10^{-5}$ & \phantom{000}0.821 \\
5 & WENO5/RK5 ($\varepsilon=10^{-10}$)
  & $\sfrac{1}{158}$  & $9.71\times 10^{-5}$ & \phantom{00}21.4\phantom{00} \\
5 & linear FD5/RK5
  & $\sfrac{1}{110}$  & $9.87\times 10^{-5}$ & \phantom{00}11.7\phantom{00} \\
\hline
\end{tabular}
\vspace{.7em}
\caption{CPU times required to achieve an $L^{\infty}$ error of less than $10^{-4}$.
Note the lower accuracy of the first order schemes.}
\label{table:efficiency}
\vspace{-1em}
\end{table}

As a second test we compare the schemes in terms of their actual efficiency. To that end, we measure the CPU times that the schemes require to achieve a given accuracy, i.e.\ for each scheme we choose the coarsest resolution such that the $L^\infty\/$ error is below $10^{-4}$, and for this resolution we report the required CPU time. The results of this test are shown in Table~\ref{table:efficiency}. Note that the first order schemes were computed at a maximum resolution of $h = \sfrac{1}{1600}\/$. A simple extrapolation reveals that in order to reach an accuracy of $10^{-4}$, CPU times of more than 50 years would be required. In contrast, the third and fifth order methods yield reasonable CPU times, except for WENO3 with $\varepsilon=10^{-10}$, which due to the degraded convergence rate requires an unreasonably fine resolution to reach the target accuracy. In fact, for this very smooth test case, linear finite differences present themselves as rather efficient methods, which achieve descent accuracies at a very low cost per time step. However, the results show that jet schemes are actually more efficient than their WENO/finite difference counterparts. Their higher cost per time step is more than outweighed by the significant increase in accuracy. Specifically, third order jet schemes are observed to be more efficient than finite differences by a factor of 1.5, and fifth order jet schemes by a factor of at least 7.5.

The computational efficiency of jet schemes goes in addition to their structural advantages, such as optimal locality and subgrid resolution. A more detailed efficiency comparison of jet schemes with WENO, and also with Discontinuous Galerkin approaches, is given in \cite{ChidyagwaiNaveRosalesSeibold2011}.


\section{Conclusions and Outlook}
\label{sec:conclusions_outlook}
The jet schemes introduced in this paper form a new class of numerical
methods for the linear advection equation. They arise from a conceptually
straightforward formalism of advect--and--project in function spaces, and
provide a systematic methodology for the construction of high
(arbitrary) order numerical schemes. They are based on two ingredients: an
ODE solver to track characteristic curves, and Hermite interpolation.
Unlike finite volume methods or discontinuous Galerkin approaches, jet
schemes have (currently) a limited area of application: linear advection
equations that are approximated on rectangular grids. However, under these
circumstances, they have several interesting advantages: a natural
treatment of the flow nature of the equation and of boundary conditions,
no CFL restrictions, no requirement for SSP ODE solvers, and optimal
locality, i.e.\ the update rule for the data on a grid point uses
information in only a single grid cell.

Jet schemes achieve high order by carrying a portion of the jet of the
solution as data. To derive an update rule in time for the data, the notion
of superconsistency is introduced. This concept provides a way 
to systematically inherit update procedures for the derivatives from an
approximation scheme (ODE solver) for the characteristic curves. Various
alternative approaches are
presented for how to track different portions of the jet of the solution,
and for how to specifically implement superconsistent schemes.
One example is the grid-based reconstruction of higher derivatives, which
on the one hand can dramatically reduce the computational cost, but on the
other hand could destabilize the scheme. Another example is the idea of
$\varepsilon$-finite differences, which requires diligent considerations
of round-off errors, but can lead to simpler implementations and also
reduce the computational cost.

The formal equivalence of superconsistent schemes to advect--and--project
approaches in function spaces gives rise to a notion of stability: the
projection step (which is based on a piece-wise application of Hermite
interpolation) minimizes a stability functional, while the increase by the
advection step can be appropriately bounded.
The stability functional gives control over certain derivatives of the
numerical solution, and thus bounds the occurrence of oscillations. In
this paper, a full proof of stability for the 1-D constant coefficient
case is provided.

Investigations of the numerical error convergence and performance of
high order jet schemes show that they tend to be more costly, but also
strikingly more accurate than classical WENO schemes of the same
respective orders. It is observed that, for the same resolution, a jet
scheme incurs 2--7 times the computational cost of a WENO scheme.
However, the numerical experiments also show that a jet scheme achieves
the same accuracy as a WENO scheme (of the same order) at a 3--4 times
coarser grid resolution. Efficiency comparisons reveal that jet schemes
tend to achieve the same accuracy as an efficient WENO scheme at
a 1.5--7 times smaller computational cost. In comparison with WENO
schemes with a sub-optimal parameter choice, the efficiency gains of
jet schemes are factors of 14--85.
An application to a model of the passive advection of contours
demonstrates the practical accuracy of jet schemes. A particular
feature of jet schemes is their potential to represent structures that
are smaller than the grid resolution.

This paper should be seen as a first step into the class of jet schemes.
Their general interpretation as advect--and--project approaches, their
conceptual simplicity, and their apparent accuracy make them promising
candidates for numerical methods for certain important types of advection
problems. However, many questions remain to be answered, and many aspects
remain to be investigated, before one could consider them generally
``competitive''.
The listing below gives a few examples of some of the aspects where further
research and development is still needed.

\textbf{Analysis.}
A general proof of stability in arbitrary space dimensions, and for
variable coefficients, is one important goal for jet schemes. A
fundamental understanding of the stability of these methods is
particularly important when combined with grid-based finite difference
approximations to higher derivatives.

\textbf{Efficiency.}
Some ideas presented here (e.g.\ grid-based or $\varepsilon$-finite
differences) lead to significant reductions in implementation effort and
computational cost, but they also indicate that the stability of such new
approaches is not automatically guaranteed. One can safely expect that
there are many other ways to implement jet schemes more easily and more
efficiently. The investigation and analysis of such ideas is an important
target for future research.

\textbf{Applicability.}
The field of linear advection problems on rectangular grids is not as
limited as it may seem. Many examples of interface tracking fall into
this area. Of particular interest is the combination of jet schemes with
adaptive mesh refinement techniques. The optimal locality of jet schemes
promises a straightforward generalization to adaptive meshes.

\textbf{Generality.}
Since all that jet schemes need are a method for tracking the
characteristics, and an interpolation procedure, they are not fundamentally
limited to rectangular grids. Similarly, they could be generalized to
allow for non-conforming boundaries, and thus apply to general
unstructured domains. Furthermore, it is plausible that the idea of
tracking derivatives can be advantageous for nonlinear Hamilton-Jacobi
equations, hyperbolic conservation laws, or diffusion problems as well.
In all these cases new challenges arise, such as the occurrence of shocks
or the absence of characteristics.

\section*{Acknowledgments}
%
The authors would like to acknowledge support by the National Science
Foundation through grant DMS--0813648. In addition,
R. R. Rosales and B. Seibold wish to acknowledge partial support by the
National Science Foundation through through grants DMS--1007967 and
DMS--1007899, as well as grants DMS--1115269 and DMS--1115278, respectively.
Further, J.-C. Nave wishes to acknowledge partial support by the NSERC
Discovery Program.



\bibliographystyle{plain}
\bibliography{references_complete}

\begin{thebibliography}{10}

\bibitem{AdalsteinssonSethian1999}
D.~Adalsteinsson and J.~A. Sethian.
\newblock The fast construction of extension velocities in level set methods.
\newblock {\em J. Comput. Phys.}, 148:2--22, 1999.

\bibitem{Arutyunov2000}
A.~V. Arutyunov.
\newblock {\em Optimality conditions: {A}bnormal and degenerate problems}.
\newblock Kluwer Academic Publishers, The Netherlands, 2000.

\bibitem{BellColellaGlaz1989}
J.~B. Bell, P.~Colella, and H.~Glaz.
\newblock A second-order projection method for the incompressible
  {N}avier-{S}tokes equations.
\newblock {\em J. Comput. Phys.}, 85:257--283, 1989.

\bibitem{BergerOliger1984}
M.~J. Berger and J.~Oliger.
\newblock Adaptive mesh refinement for hyperbolic partial differential
  equations.
\newblock {\em J. Comput. Phys.}, 53:484--512, 1984.

\bibitem{CashKarp1990}
J.~R. Cash and A.~H. Karp.
\newblock A variable order {R}unge-{K}utta method for initial value problems
  with rapidly varying right-hand sides.
\newblock {\em ACM T. Math. Software}, 16:201--222, 1990.

\bibitem{ChidyagwaiNaveRosalesSeibold2011}
P.~Chidyagwai, J.-C. Nave, R.~R. Rosales, and B.~Seibold.
\newblock A comparative study of the efficiency of jet schemes.
\newblock {\em Int. J. Numer. Anal. Model.-B}, (under review), 2011.

\bibitem{CockburnShu1988}
B.~Cockburn and C.-W. Shu.
\newblock The local {D}iscontinuous {G}alerkin method for time-dependent
  convection-diffusion systems.
\newblock {\em SIAM J. Numer. Anal.}, 35(6):2440--2463, 1988.

\bibitem{CourantIsaacsonRees1952}
R.~Courant, E.~Isaacson, and M.~Rees.
\newblock On the solution of nonlinear hyperbolic differential equations by
  finite differences.
\newblock {\em Comm. Pure Appl. Math.}, 5:243--255, 1952.

\bibitem{Gottlieb2005}
S.~Gottlieb.
\newblock On high order strong stability preserving {R}unge-{K}utta and multi
  step time discretizations.
\newblock {\em Journ. Scientif. Computing}, 25(1):105--128, 2005.

\bibitem{GottliebKetchesonShu2009}
S.~Gottlieb, D.~I. Ketcheson, and C.-W. Shu.
\newblock High order strong stability preserving time discretizations.
\newblock {\em Journ. Scientif. Computing}, 38(3):251--289, 2009.

\bibitem{GottliebShu1998}
S.~Gottlieb and C.-W. Shu.
\newblock Total variation diminishing {R}unge-{K}utta schemes.
\newblock {\em Math. Comp.}, 67(221):73--85, 1998.

\bibitem{GottliebShuTadmor2001}
S.~Gottlieb, C.-W. Shu, and E.~Tadmor.
\newblock Strong stability preserving high order time discretization methods.
\newblock {\em SIAM Review}, 43(1):89--112, 2001.

\bibitem{Heller1960}
J.~P. Heller.
\newblock An unmixing demonstration.
\newblock {\em Am. J. Phys.}, 28:348--353, 1960.

\bibitem{HesthavenWarburton2008}
W.~Hesthaven and T.~Warburton.
\newblock {\em Nodal {D}iscontinuous {G}alerkin methods: {A}lgorithms,
  analysis, and applicationss}, volume~54 of {\em Texts in Applied
  Mathematics}.
\newblock Springer, New York, 2008.

\bibitem{JiangShu1996}
G.-S. Jiang and C.-W. Shu.
\newblock Efficient implementation of weighted {ENO} schemes.
\newblock {\em J. Comput. Phys.}, 126(1):202--228, 1996.

\bibitem{Kontsevich1995}
M.~L. Kontsevich.
\newblock Lecture at {O}rsay, December 1995.

\bibitem{LeVeque1996}
R.~LeVeque.
\newblock High-resolution conservative algorithms for advection in
  incompressible flow.
\newblock {\em SIAM J. Numer. Anal.}, 33:627--665, 1996.

\bibitem{LiuOsherChan1994}
X.-D. Liu, S.~Osher, and T.~Chan.
\newblock Weighted essentially non-oscillatory schemes.
\newblock {\em J. Comput. Phys.}, 115:200--212, 1994.

\bibitem{MinGibou2007}
C.-H. Min and F.~Gibou.
\newblock A second order accurate level set method on non-graded adaptive
  cartesian grids.
\newblock {\em J. Comput. Phys.}, 225(1):300--321, 2007.

\bibitem{Nash1995}
J.~F. Nash~Jr.
\newblock Arc structure of singularities.
\newblock {\em Duke Math. J.}, 81(1):31--38, 1995.
\newblock (Written in 1966).

\bibitem{NaveRosalesSeibold2010}
J.-C. Nave, R.~R. Rosales, and B.~Seibold.
\newblock A gradient-augmented level set method with an optimally local,
  coherent advection scheme.
\newblock {\em J. Comput. Phys.}, 229:3802--3827, 2010.

\bibitem{OsherSethian1988}
S.~Osher and J.~A. Sethian.
\newblock Fronts propagating with curvature--dependent speed: {A}lgorithms
  based on {H}amilton--{J}acobi formulations.
\newblock {\em J. Comput. Phys.}, 79:12--49, 1988.

\bibitem{ReedHill1973}
W.~H. Reed and T.~R. Hill.
\newblock Triangular mesh methods for the neutron transport equation.
\newblock Technical Report LA-UR-73-479, Los Alamos Scientific Laboratory,
  1973.

\bibitem{Shu1988}
C.-W. Shu.
\newblock Total-variation diminishing time discretizations.
\newblock {\em SIAM J. Sci. Stat. Comp.}, 9(6):1073--1084, 1988.

\bibitem{ShuOsher1988}
C.-W. Shu and S.~Osher.
\newblock Efficient implementation of essentially non-oscillatory
  shock-capturing schemes.
\newblock {\em J. Comput. Phys.}, 77:439--471, 1988.

\bibitem{SussmanSmerekaOsher1994}
M.~Sussman, P.~Smereka, and S.~Osher.
\newblock A level set approach for computing solutions to incompressible
  two-phase flow.
\newblock {\em J. Comput. Phys.}, 114(1):146--159, 1994.

\bibitem{FluidMixingMovie}
G.~I. Taylor.
\newblock Low {R}eynolds number flow.
\newblock Movie, 1961.
\newblock U.S. National Committee for Fluid Mechanics Films (NCFMF).

\bibitem{VanLeer1973}
B.~van Leer.
\newblock Towards the ultimate conservative difference scheme {I}. {T}he quest
  of monoticity.
\newblock {\em Springer Lecture Notes in Physics}, 18:163--168, 1973.

\bibitem{VanLeer1979}
B.~van Leer.
\newblock Towards the ultimate conservative difference scheme {V}. {A}
  second-order sequel to {G}odunov's method.
\newblock {\em J. Comput. Phys.}, 32:101--136, 1979.

\end{thebibliography}

\end{document}